\documentclass[11pt,a4paper]{article}
\usepackage[T1]{fontenc}
\usepackage[latin1]{inputenc}
\usepackage{amsmath}
\usepackage{amssymb}
\usepackage{color}
\usepackage{fancyhdr}
\usepackage[top=28mm,bottom=33mm,left=21mm,right=21mm]{geometry}
\usepackage{paralist}
\usepackage{amsthm}
\usepackage{mathrsfs}
\usepackage{yfonts}
\usepackage{framed}
\usepackage{pifont}
\usepackage{textcomp}
\usepackage{array}
\usepackage{graphicx}
\usepackage{wrapfig}
\usepackage{enumitem}
\usepackage{hyperref}
\usepackage{pstricks}
\usepackage{pst-plot}

\title{Self-destructive percolation as a limit of forest-fire models on regular rooted trees}
\author{Robert Graf \\ \small \emph{Mathematisches Institut, Ludwig-Maximilians-Universit\"at M\"unchen} \\ \small \emph{Theresienstr. 39, 80333 M\"unchen, Germany} \\ \small \texttt{robert.graf@math.lmu.de}}
\date{}

\theoremstyle{plain}
\newtheorem{thm}{Theorem}
\newtheorem{lem}{Lemma}
\newtheorem{cor}{Corollary}
\newtheorem{prop}{Proposition}

\theoremstyle{definition}
\newtheorem{defn}{Definition}

\newcommand{\mb}{\mathbb}
\newcommand{\mc}{\mathcal}
\newcommand{\mf}{\mathbf}

\newcommand{\keywords}{\textbf{Key words. }\medskip}
\newcommand{\subjclass}{\textbf{MSC 2010. }\medskip}

\begin{document}

\maketitle

\begin{abstract}
Let $T$ be a regular rooted tree. For every natural number $n$, let $B_n$ be the finite subtree of vertices with graph distance at most $n$ from the root. Consider the following forest-fire model on $B_n$: Each vertex can be ``vacant'' or ``occupied''. At time $0$ all vertices are vacant. Then the process is governed by two opposing mechanisms: Vertices become occupied at rate $1$, independently for all vertices. Independently thereof and independently for all vertices, ``lightning'' hits vertices at rate $\lambda(n) > 0$. When a vertex is hit by lightning, its occupied cluster instantaneously becomes vacant.

Now suppose that $\lambda(n)$ decays exponentially in $n$ but much more slowly than $1/|B_n|$. We show that then there exist a supercritical time $\tau$ and $\epsilon > 0$ such that the forest-fire model on $B_n$ between time $0$ and time $\tau + \epsilon$ tends to the following process on $T$ as $n$ goes to infinity: At time $0$ all vertices are vacant. Between time $0$ and time $\tau$ vertices become occupied at rate $1$, independently for all vertices. At time $\tau$ all infinite occupied clusters become vacant. Between time $\tau$ and time $\tau + \epsilon$ vertices again become occupied at rate $1$, independently for all vertices. At time $\tau + \epsilon$ all occupied clusters are finite. This process is a dynamic version of self-destructive percolation.
\end{abstract}



\section{Introduction and statement of the main results} \label{sec introduction}

\subsection{Introduction to forest-fire models}

Forest-fire models were first introduced in the physics literature by B.\ Drossel and F.\ Schwabl (see \cite{drossel1992forest}) as an example of self-organized criticality and have recently been studied by various mathematicians. Put simply, the notion of self-organized criticality is used for dynamical systems with local interactions which are inherently driven towards a perpetual critical state. At the critical state, the local interactions build up to trigger global ``catastrophic'' events, which are characterized by power laws, self-similarity and fractal behaviour. For a detailed account of self-organized criticality, the reader is referred to \cite{bak1999nature} and \cite{jensen1998self}. In this paper we consider a version of the forest-fire model which is defined on regular rooted trees or, more precisely, large finite subtrees thereof.

Let us start by introducing some notation about regular rooted trees. For the remainder of this paper, let $r \in \{2, 3, \ldots\}$ be fixed. The \textbf{$r$-regular rooted tree} is the unique tree (up to graph isomorphisms) in which one vertex, called the \textbf{root} of the tree, has degree $r$ and every other vertex has degree $r + 1$. We denote the $r$-regular tree by $T$ and the root of $T$ by $\emptyset$. In slight abuse of notation, we will use the term $T$ both for the $r$-regular tree as a graph and for its vertex set. Let $|u|$ denote the graph distance of a vertex $u \in T$ from the root $\emptyset$. For two vertices $u, v \in T$, we say that $u$ is the \textbf{parent} of $v$ (or equivalently that $v$ is a \textbf{child} of $u$) if $u$ and $v$ are neighbours and $|u| = |v| - 1$ holds. Moreover, for $u, v \in T$, we say that $u$ is an \textbf{ancestor} of $v$ (abbreviated by $u \preceq v$) if there exist $k \in \mb{N}_0$ and a sequence of vertices $z_0, z_1, \ldots, z_k \in T$ such that $z_0 = u$, $z_k = v$ and $z_{i - 1}$ is the parent of $z_i$ for all $i \in \{1, 2, \ldots, k\}$. For $n \in \mb{N}_0$, we say that $u \in T$ is in the \textbf{$n$th generation} of $T$ if $|u| = n$, and we define
\begin{align*}
T_n := \left\{ z \in T: |z| = n \right\}
\end{align*}
to be the set of all vertices in the $n$th generation and
\begin{align*}
B_n := \left\{ z \in T: |z| \leq n \right\} = \bigcup_{i = 0}^n T_i
\end{align*}
to be the set of all vertices with graph distance at most $n$ from the root $\emptyset$.

In order to explain some further terminology, let $V$ be a subset of $T$ and let $\alpha = (\alpha_v)_{v \in V} \in \{0, 1\}^{V}$. We say that a vertex $v \in V$ is occupied in $\alpha$ if $\alpha_v = 1$, and we say that $v$ is vacant in $\alpha$ if $\alpha_v = 0$. The set
\begin{align*}
T|_{\alpha, 1} := \left\{ v \in V: \alpha_v = 1 \right\} \subset T
\end{align*}
of occupied vertices in $\alpha$ induces a subgraph of $T$, which (in slight abuse of notation) we denote by $T|_{\alpha, 1}$, too. For any vertex $z \in V$ the maximal connected component of $T|_{\alpha, 1}$ containing $z$ is called the \textbf{(occupied) cluster} of $z$ in $\alpha$. Moreover, if $W$ is a connected subset of $T$, we say that a vertex $z \in T$ is the \textbf{root} of $W$ if $z \in W$ holds and $z$ is in the lowest generation among all vertices contained in $W$.

Let $n \in \mb{N}$. We now define the forest-fire model on $B_n$. Informally, the model can be described as follows: Each vertex in $B_n$ can be vacant or occupied. At time $0$ all vertices are vacant. Then the process is governed by two opposing mechanisms: Vertices become occupied according to independent rate $1$ Poisson processes, the so-called growth processes. Independently, vertices are hit by ``lightning'' according to independent rate $\lambda(n)$ Poisson processes (where $\lambda(n) > 0$), the so-called ignition processes. When a vertex is hit by lightning, its occupied cluster is instantaneously destroyed, i.e.\ it becomes vacant. Occupied vertices are usually pictured to be vegetated by a tree, so occupied clusters correspond to pieces of woodland and the destruction of clusters corresponds to the burning of forests by fires, which are caused by strokes of lightning. However, we avoid this terminology here because we already use the term tree in the graph-theoretic sense. A more formal definition of the forest-fire model goes as follows (where for a function $[0, \infty) \ni t \mapsto f_t \in \mb{R}$, we write $f_{t^-} := \lim_{s \uparrow t} f_s$ for the left-sided limit at $t > 0$, provided the limit exists):

\begin{defn} \label{defn forest-fire model}
Let $n \in \mb{N}$ and $\lambda(n) > 0$. Let $(\eta^n_{t, z}, G_{t, z}, I^n_{t, z})_{t \geq 0, z \in B_n}$ be a process with values in $(\{0, 1\} \times \mb{N}_0 \times \mb{N}_0)^{[0, \infty) \times B_n}$ and initial condition $\eta^n_{0, z} = 0$ for $z \in B_n$. Suppose that for all $z \in B_n$ the process $(\eta^n_{t, z}, G_{t, z}, I^n_{t, z})_{t \geq 0}$ is càdlàg, i.e.\ right-continuous with left limits. For $z \in B_n$ and $t > 0$, let $C^n_{t^-, z}$ denote the cluster of $z$ in the configuration $(\eta^n_{t^-, w})_{w \in B_n}$.

Then $(\eta^n_{t, z}, G_{t, z}, I^n_{t, z})_{t \geq 0, z \in B_n}$ is called a \textbf{forest-fire process} on $B_n$ with parameter $\lambda(n)$ if the following conditions are satisfied:
\begin{description}[leftmargin=3.3cm,style=sameline,font=\normalfont]
\item[\text{[POISSON]}] The processes $(G_{t, z})_{t \geq 0}$ and $(I^n_{t, z})_{t \geq 0}$, $z \in B_n$, are independent Poisson processes with rates $1$ and $\lambda(n)$, respectively.
\item[\text{[GROWTH]}] For all $t > 0$ and all $z \in B_n$ the following implications hold:
\begin{compactenum}[(i)]
\item $G_{t^-, z} < G_{t, z} \Rightarrow \eta^n_{t, z} = 1$, \\
i.e.\ the growth of a tree at the site $z$ at time $t$ implies that the site $z$ is occupied at time $t$;
\item $\eta^n_{t^-, z} < \eta^n_{t, z} \Rightarrow G_{t^-, z} < G_{t, z}$, \\
i.e.\ if the site $z$ gets occupied at time $t$, there must have been the growth of a tree at the site $z$ at time $t$.
\end{compactenum}
\item[\text{[DESTRUCTION]}] For all $t > 0$ and all $z \in B_n$ the following implications hold:
\begin{compactenum}[(i)]
\item $I^n_{t^-, z} < I^n_{t, z} \Rightarrow \forall w \in C^n_{t^-, z}: \eta^n_{t, w} = 0$, \\
i.e.\ if the cluster at $z$ is hit by lightning at time $t$, it is destroyed at time $t$;
\item $\eta^n_{t^-, z} > \eta^n_{t, z} \Rightarrow \exists v \in C^n_{t^-, z}: I^n_{t^-, v} < I^n_{t, v}$, \\
i.e.\ if the site $z$ is destroyed at time $t$, its cluster must have been hit by lightning at time $t$.
\end{compactenum}
\end{description}
\end{defn}

Given independent Poisson processes $(G_{t, z})_{t \geq 0}$ and $(I^n_{t, z})_{t \geq 0}$, $z \in B_n$, with rates $1$ and $\lambda(n)$, respectively, a unique corresponding forest-fire process $(\eta^n_{t, z}, G_{t, z}, I^n_{t, z})_{t \geq 0, z \in B_n}$ on $B_n$ can be obtained by a graphical construction (see \cite{liggett2004interacting}). For this construction it is crucial that $B_n$ is finite. Using different methods, M.\ Dürre obtained results on existence and uniqueness of forest-fire models for all connected infinite graphs with bounded vertex degree (see \cite{duerre2006existence}, \cite{duerre2006uniqueness}, \cite{duerre2009thesis}).

One of the most interesting aspects about the forest-fire process on $B_n$ is the question of what happens when $n$ tends to infinity. Assuming that the limit $n \rightarrow \infty$ exists in a suitable sense, we obtain a process on the infinite tree $T$, and the question thus concerns the dynamics of this limit process. It is intuitively clear that the growth mechanism carries over to the limit process but it is in general highly non-trivial what becomes of the destruction mechanism. Of course, the answer will depend strongly on the asymptotic behaviour of $\lambda(n)$. If $a$, $b$ are functions from $\mb{N}$ to $(0, \infty)$, we write
\begin{compactenum}[(i)]
\item $a(n) \ll b(n)$ for $n \rightarrow \infty$ if $a(n)/b(n) \rightarrow 0$ for $n \rightarrow \infty$;
\item $a(n) \approx b(n)$ for $n \rightarrow \infty$ if $\log a(n)/\log b(n) \rightarrow 1$ for $n \rightarrow \infty$.
\end{compactenum}
Heuristically, one expects four regimes of $\lambda(n)$ with qualitatively different asymptotics, which we now describe informally.

\begin{enumerate}
\item If $\lambda(n) \ll r^n$, then the number of lightnings in $B_n$ tends to $0$ for $n \rightarrow \infty$. Therefore, in the limit $n \rightarrow \infty$ no clusters can ever be destroyed so that the resulting process on $T$ is simply a dynamical formulation of Bernoulli percolation.
\item If $\lambda(n) \approx 1/m^n$ for some $1 < m < r$, then in the limit $n \rightarrow \infty$ no finite clusters and no ``thin'' infinite clusters (i.e.\ those in which on average every vertex has fewer then $m$ occupied child vertices) can be destroyed but ``fat'' infinite clusters (i.e.\ those in which on average every vertex has more then $m$ occupied child vertices) should still be hit by lightning as soon as they appear. The resulting process on $T$ should therefore have the following dynamics: Vertices become occupied at rate $1$, independently for all vertices. If an infinite cluster becomes ``fat'', it is instantaneously destroyed.
\item If $1/m^n \ll \lambda(n) \ll 1$ for every $m > 1$, then in the limit $n \rightarrow \infty$ no finite clusters can be destroyed but one would expect any infinite cluster to be dense enough that it is hit by lightning as soon as it appears. The resulting process on $T$ should therefore have the following dynamics: Vertices become occupied at rate $1$, independently for all vertices. If a cluster becomes infinite, it is instantaneously destroyed.
\item If $\lambda(n) = \lambda$ for some constant $\lambda > 0$, then the limit $n \rightarrow \infty$ should yield a forest-fire model on $T$ with the following dynamics: Vertices become occupied at rate $1$, independently for all vertices. Independently thereof and independently for all vertices, vertices are hit by lightning at rate $\lambda$. If a vertex is hit by lightning, its cluster is instantaneously destroyed.
\end{enumerate}

In this paper, we give a partial result for regime 2 in the sense that we prove the conjectured asymptotics between time $0$ and a deterministic time shortly \emph{after} the first destruction of infinite clusters in the limit process on $T$. Before we proceed to the precise statement, we briefly comment on the other regimes and give a short overview of related results.

Regime 1 is the simplest case and the above statement on this regime can easily be made rigorous. The statement on regime 4 follows from work by M.\ Dürre in \cite{duerre2009thesis}. In fact, the results of \cite{duerre2009thesis} are much more general in the sense that they are not restricted to regular rooted trees but hold for all connected infinite graphs with bounded vertex degree. Regime 3 is undoubtedly the most difficult case with few rigorous results yet. It is even unknown whether the hypothetical limit process described in 3 exists at all. For the square lattice $\mb{Z}^2$, the corresponding process does not exist (conjectured by J.\ van den Berg and R.\ Brouwer in \cite{berg2004destructive} and recently proven by D.\ Kiss, I.\ Manolescu and V.\ Sidoravicius in \cite{kiss2013planar}). Regime 3 is expected to behave similarly to the case where we first set $\lambda(n) = \lambda$ for some $\lambda > 0$ and then take the double limit $\lim_{\lambda \downarrow 0} \lim_{n \rightarrow \infty}$ (assuming that it exists in a suitable sense). In \cite{berg2006critical} this case was investigated for forest-fire models on the directed binary tree and on the square lattice. For forest-fire models on the square lattice $\mb{Z}^2$, an analogous heuristic description of four different regimes of the lightning rate can be found in the paper \cite{rath2006erdos} by B.\ Ráth and B.\ Tóth. The main content of \cite{rath2006erdos}, however, is the analysis of a forest-fire model which arises as a modification of the Erd\H os-Rényi evolution and which also shows four regimes of the lightning rate with different asymptotic behaviour.

\subsection{The pure growth process}

In the following, if $A$ is an event, we write $1_A$ for its indicator function, and if $B$ is any set, we write $|B|$ for the number of elements in $B$ (where $|B|$ can take values in $\mb{N}_0 \cup \{\infty\}$).

\begin{defn} \label{defn pure growth process}
Let $(G_{t, z})_{t \geq 0}$, $z \in T$, be independent rate $1$ Poisson processes and let
\begin{align*}
\sigma_{t, z} := 1_{\{G_{t, z} > 0\}} \text{,} \qquad t \geq 0, z \in T \text{.}
\end{align*}
Then $(\sigma_{t, z}, G_{t, z})_{t \geq 0, z \in T}$ is called a \textbf{pure growth process} on $T$. Moreover, for $x \in T$ and $t \geq 0$, we denote by $S_{t, x}$ the cluster of $x$ in the configuration $(\sigma_{t, z})_{z \in T}$, and for $t \geq 0$, we denote by
\begin{align*}
O_t := \left\{ z \in T: |S_{t, z}| = \infty \right\}
\end{align*}
the set of all vertices which are in an infinite cluster in the configuration $(\sigma_{t, z})_{z \in T}$.
\end{defn}

Above we claimed that as $n \rightarrow \infty$ in regime 2, the forest-fire process on $B_n$ should initially behave like the pure growth process on $T$ until ``fat'' infinite clusters appear for the first time. We now want to make this statement more precise.

We first observe that for $t \geq 0$, the configuration $(\sigma_{t, z})_{z \in T}$ is identical with Bernoulli percolation on $T$, where each vertex is occupied with probability $1 - e^{-t}$ and vacant with probability $e^{-t}$. From percolation theory it is well-known that there is a critical time $t_c := \log \frac{r}{r - 1}$ such that a.s.\ for $t \leq t_c$ there is no infinite cluster in $(\sigma_{t, z})_{z \in T}$ and for $t > t_c$ there are infinitely many infinite clusters in $(\sigma_{t, z})_{z \in T}$. For $z \in T$ and $t \geq 0$, conditionally on the event $\left\{ z \text{ is the root of } S_{t, z} \right\}$, the cluster $S_{t, z}$ can also be identified with a Galton-Watson process whose offspring distribution is binomially distributed with parameters $r$ and $1 - e^{-t}$. In particular, the offspring distribution at time $t \geq 0$ has mean
\begin{align} \label{eq mean binomial dist}
m(t) := r(1 - e^{-t})
\end{align}
and variance
\begin{align} \label{eq variance binomial dist}
\sigma^2(t) := r(1 - e^{-t})e^{-t} \text{.}
\end{align}
It is a consequence of the Kesten-Stigum theorem for Galton-Watson processes (see \cite{kesten1966stigum}) that for $z \in T$ and $t \geq 0$, there exists a random variable $W_{t, z}$ with values in $[0, \infty)$ such that
\begin{align} \label{eq Kesten-Stigum 1}
\lim_{n \rightarrow \infty} \frac{|S_{t, z} \cap B_n|}{m(t)^n} = W_{t, z} \text{ a.s.}
\end{align}
and
\begin{align} \label{eq Kesten-Stigum 2}
W_{t, z} > 0 \text{ a.s.\ on the event } \left\{ |S_{t, z}| = \infty \right\}
\end{align}
hold. (We will prove a different version later, see Proposition \ref{prop replacement Kesten-Stigum}.) This suggests that if the lightning rate in the forest-fire process on $B_n$ satisfies $\lambda(n) \approx 1/m^n$ for some $1 < m < r$, then the time threshold between ``thin'' and ``fat'' infinite clusters in the pure growth process should be the unique $\tau \in (t_c, \infty)$ with $m(\tau) = m$. In other words, in the limit $n \rightarrow \infty$, we expect to obtain a process on $T$ which is equal to the pure growth process between time $0$ and time $\tau$ and in which all infinite clusters are destroyed at time $\tau$.

\subsection{Statement of the main results}

We will make the heuristics of the previous paragraph rigorous in the following way:

\begin{defn}
Let $n \in \mb{N}$ and let $\lambda(n) > 0$. We say that a forest-fire process $(\eta^n_{t, z}, G_{t, z}, I^n_{t, z})_{t \geq 0, z \in B_n}$ on $B_n$ with parameter $\lambda(n)$ and a pure growth process $(\tilde \sigma_{t, z}, \tilde G_{t, z})_{t \geq 0, z \in T}$ on $T$ are \textbf{coupled in the canonical way} if they are realized on the same probability space and $(G_{t, z})_{t \geq 0, z \in B_n} = (\tilde G_{t, z})_{t \geq 0, z \in B_n}$ holds.
\end{defn}

\begin{thm} \label{thm general convergence}
Let $\tau \in (t_c, \infty)$ and suppose that $\lambda: \mb{N} \rightarrow (0, \infty)$ satisfies $\lambda(n) \approx 1/m(\tau)^n$ for $n \rightarrow \infty$. For $n \in \mb{N}$, let $(\eta^n_{t, z}, G_{t, z}, I^n_{t, z})_{t \geq 0, z \in B_n}$ be a forest-fire process on $B_n$ with parameter $\lambda(n)$ and let $(\sigma_{t, z}, G_{t, z})_{t \geq 0, z \in T}$ be a pure growth process on $T$, coupled in the canonical way under some probability measure $\mf{P}$. For $t \geq 0$, let $O_t$ be defined as in Definition \ref{defn pure growth process}. Then for all finite subsets $E \subset T$ and for all $\delta > 0$,
\begin{align*}
\lim_{n \rightarrow \infty} \mf{P} \left[ \sup_{z \in E, 0 \leq t \leq \tau - \delta} \left| \eta^n_{t, z} - \sigma_{t, z} \right| = 0, \forall z \in O_{\tau} \cap E \, \exists t \in (\tau - \delta, \tau + \delta): \eta^n_{t^-, z} > \eta^n_{t, z} \right] = 1
\end{align*}
holds.
\end{thm}

The condition on $\lambda$ in Theorem \ref{thm general convergence} can be written in a different way: Given $\tau \in (t_c, \infty)$ and a function $\lambda: \mb{N} \rightarrow (0, \infty)$, define the function $g: \mb{N} \rightarrow (0, \infty)$ by
\begin{align} \label{equation g}
g(n) := \lambda(n) m(\tau)^n \text{,} \qquad n \in \mb{N} \text{.}
\end{align}
Then it is easy to see that the following are equivalent:
\begin{compactenum}[(i)]
\item $\lambda(n) \approx 1/m(\tau)^n$ for $n \rightarrow \infty$;
\item $\sqrt[n]{g(n)} \rightarrow 1$ for $n \rightarrow \infty$.
\end{compactenum}
Under additional assumptions on $g$ we can determine whether the destruction of the infinite clusters asymptotically occurs immediately \emph{before} or \emph{after} time $\tau$:

\begin{thm} \label{thm refined convergence}
Consider the situation of Theorem \ref{thm general convergence}. In particular, suppose that the function $g$ defined by (\ref{equation g}) satisfies $\sqrt[n]{g(n)} \rightarrow 1$ for $n \rightarrow \infty$.
\begin{compactenum}[(i)]
\item If $g$ satisfies $g(n) \ll n/\log n$ for $n \rightarrow \infty$, then for all finite subsets $E \subset T$ and for all $\delta > 0$,
\begin{align*}
\lim_{n \rightarrow \infty} \mf{P} \left[ \sup_{z \in E, 0 \leq t \leq \tau} \left| \eta^n_{t, z} - \sigma_{t, z} \right| = 0, \forall z \in O_{\tau} \cap E \, \exists t \in (\tau, \tau + \delta): \eta^n_{t^-, z} > \eta^n_{t, z} \right] = 1
\end{align*}
holds, i.e.\ the infinite clusters are asymptotically destroyed immediately \emph{after} time $\tau$.
\item If there exists $\alpha \in (0, 1)$ such that $g$ satisfies $g(n) \gg \exp(n^{\alpha})$ for $n \rightarrow \infty$, then for all finite subsets $E \subset T$ and for all $\delta > 0$,
\begin{align*}
\lim_{n \rightarrow \infty} \mf{P} \left[ \sup_{z \in E, 0 \leq t \leq \tau - \delta} \left| \eta^n_{t, z} - \sigma_{t, z} \right| = 0, \forall z \in O_{\tau} \cap E \, \exists t \in (\tau - \delta, \tau): \eta^n_{t^-, z} > \eta^n_{t, z} \right] = 1
\end{align*}
holds, i.e.\ the infinite clusters are asymptotically destroyed immediately \emph{before} time $\tau$.
\end{compactenum}
\end{thm}

Theorems \ref{thm general convergence} and \ref{thm refined convergence} will be proved in Sections \ref{sec proof general convergence} and \ref{sec proof refined convergence}. Before, we give an interpretation of Theorem~\ref{thm general convergence} in terms of self-destructive percolation.

\subsection{Interpretation in terms of self-destructive percolation}

\begin{defn} \label{defn self-destructive percolation process}
Let $\tau \in (t_c, \infty)$, let $\epsilon > 0$ and let $(\sigma_{t, z}, G_{t, z})_{t \geq 0, z \in T}$ be a pure growth process on $T$. For $t \geq 0$, let $O_t$ be defined as in Definition \ref{defn pure growth process}. We define $\rho_{t, z}$ for $0 \leq t \leq \tau + \epsilon, z \in T$ in three steps: \\
Firstly,
\begin{align*}
\rho_{t, z} &:= \sigma_{t, z} \text{,} \qquad 0 \leq t < \tau, z \in T \text{,} \\
\intertext{i.e.\ at time $0$ all vertices are vacant and between time $0$ and time $\tau$ vertices become occupied at rate~$1$, independently for all vertices.
Secondly,}
\rho_{\tau, z} &:= \sigma_{\tau, z} \, 1_{\{z \not\in O_{\tau}\}} \text{,} \qquad z \in T \text{,} \\
\intertext{i.e.\ at time $\tau$ all infinite occupied clusters are destroyed.
Thirdly,}
\rho_{t, z} &:= \rho_{\tau, z} \vee 1_{\{G_{t, z} - G_{\tau, z} > 0\}} \text{,} \qquad \tau < t \leq \tau + \epsilon, z \in T \text{,}
\end{align*}
i.e.\ between time $\tau$ and time $\tau + \epsilon$ vertices become occupied at rate $1$, independently for all vertices and independently of what happened between time $0$ and time $\tau$.
Then $(\rho_{t, z}, G_{t, z})_{0 \leq t \leq \tau + \epsilon, z \in T}$ is called a \textbf{self-destructive percolation process} on $T$ with parameters $\tau$ and $\epsilon$.
\end{defn}

Self-destructive percolation was first introduced by J.\ van den Berg and R.\ Brouwer in \cite{berg2004destructive} and has subsequently also been studied in \cite{berg2008continuity} \cite{berg2009bounds} \cite{ahlberg2013nonamenable} and \cite{ahlberg2013seven}. For our purposes, the following property of self-destructive percolation is of particular importance:

\begin{prop} \label{prop critical epsilon}
For all $\tau \in (t_c, \infty)$ there exists $\epsilon > 0$ such that a.s.\ there is no infinite cluster in the final configuration $(\rho_{\tau + \epsilon, z})_{z \in T}$ of a self-destructive percolation process $(\rho_{t, z}, G_{t, z})_{0 \leq t \leq \tau + \epsilon, z \in T}$ on $T$ with parameters $\tau$ and $\epsilon$.
\end{prop}

For the case where $T$ is the binary tree (i.e.\ $r = 2$), this has already been proved by J.\ van den Berg and R.\ Brouwer (\cite{berg2004destructive}, Theorem 5.1). The proof of Proposition~\ref{prop critical epsilon} for general $r$ is based on an extension of the ideas in \cite{berg2004destructive} and will be given in Section \ref{sec proof critical epsilon}.

Theorem \ref{thm general convergence} and Proposition \ref{prop critical epsilon} imply that given $\tau \in (t_c, \infty)$, we can choose $\epsilon > 0$ such that between time $0$ and time $\tau + \epsilon$ every forest-fire process on $B_n$ with parameter $\lambda(n) \approx 1/m(\tau)^n$ converges to the self-destructive percolation process on $T$ with parameters $\tau$ and $\epsilon$. The formal statement is as follows:

\begin{defn}
Let $n \in \mb{N}$ and let $\lambda(n) > 0$. Moreover, let $\tau \in (t_c, \infty)$ and $\epsilon > 0$. We say that a forest-fire process $(\eta^n_{t, z}, G_{t, z}, I^n_{t, z})_{t \geq 0, z \in B_n}$ on $B_n$ with parameter $\lambda(n)$ and a self-destructive percolation process $(\tilde \rho_{t, z}, \tilde G_{t, z})_{0 \leq t \leq \tau + \epsilon, z \in T}$ on $T$ with parameters $\tau$ and $\epsilon$ are \textbf{coupled in the canonical way} if they are realized on the same probability space and $(G_{t, z})_{0 \leq t \leq \tau + \epsilon, z \in B_n} = (\tilde G_{t, z})_{0 \leq t \leq \tau + \epsilon, z \in B_n}$ holds.
\end{defn}

\begin{cor} \label{cor general convergence}
Let $\tau \in (t_c, \infty)$, let $\epsilon > 0$ be as in Proposition \ref{prop critical epsilon} and suppose that $\lambda: \mb{N} \rightarrow (0, \infty)$ satisfies $\lambda(n) \approx 1/m(\tau)^n$ for $n \rightarrow \infty$. For $n \in \mb{N}$, let $(\eta^n_{t, z}, G_{t, z}, I^n_{t, z})_{t \geq 0, z \in B_n}$ be a forest-fire process on $B_n$ with parameter $\lambda(n)$ and let $(\rho_{t, z}, G_{t, z})_{0 \leq t \leq \tau + \epsilon, z \in T}$ be a self-destructive percolation process on $T$, coupled in the canonical way under some probability measure $\mf{P}$. Then for all finite subsets $E \subset T$ and for all $\delta \in (0, \epsilon)$,
\begin{align} \label{cor general convergence eq statement}
\lim_{n \rightarrow \infty} \mf{P} \left[ \sup_{z \in E, 0 \leq t \leq \tau - \delta} \left| \eta^n_{t, z} - \rho_{t, z} \right| = 0, \sup_{z \in E, \tau + \delta \leq t \leq \tau + \epsilon} \left| \eta^n_{t, z} - \rho_{t, z} \right| = 0 \right] = 1
\end{align}
holds.
\end{cor}

\begin{proof}[Proof of Corollary \ref{cor general convergence} given Theorem~\ref{thm general convergence} and Proposition~\ref{prop critical epsilon}]
Let $\tau$, $\epsilon$, $\lambda$ be as in Corollary~\ref{cor general convergence}. Likewise, for $n \in \mb{N}$, let $(\eta^n_{t, z}, G_{t, z}, I^n_{t, z})_{t \geq 0, z \in B_n}$, $(\rho_{t, z}, G_{t, z})_{0 \leq t \leq \tau + \epsilon, z \in T}$ be as in Corollary~\ref{cor general convergence}. Moreover, let $E \subset T$ be a finite subset and let $\delta \in (0, \epsilon)$. For the proof of (\ref{cor general convergence eq statement}) we may assume without loss of generality that $E$ is a singleton, i.e.\ $E = \{x\}$ for some $x \in T$. In view of Theorem \ref{thm general convergence} it then suffices to prove
\begin{align} \label{cor general convergence eq statement a}
\lim_{n \rightarrow \infty} \mf{P} \left[ \sup_{\tau + \delta \leq t \leq \tau + \epsilon} \left| \eta^n_{t, x} - \rho_{t, x} \right| = 0 \right] = 1 \text{.}
\end{align}

Before we continue with the proof, let us introduce some notation: For a non-empty subset $S \subset T$, let
\begin{align*}
\partial S := \left\{ z \in T \setminus S: \left( \exists w \in S: \text{$z$ and $w$ are neighbours} \right)  \right\}
\end{align*}
be the boundary of $S$ in $T$. For $t \in [0, \tau + \epsilon]$ and $z \in T$, let $R_{t, z}$ denote the cluster of $z$ in the configuration $(\rho_{t, w})_{w \in T}$ and let
\begin{align*}
\overline{R}_{t, z} :=
\begin{cases}
R_{t, z} \cup \partial R_{t, z} & \text{if } R_{t, z} \not= \emptyset \text{,} \\
\{z\} & \text{if } R_{t, z} = \emptyset \text{,}
\end{cases}
\end{align*}
be its ``closure''. For $t \in [0, \tau + \epsilon]$, $z \in T$ and $n \in \mb{N}$ we similarly write $C^n_{t, z}$ for the cluster of $z$ in the configuration $(\eta^n_{t, w})_{w \in B_n}$ and define its closure by
\begin{align*}
\overline{C}^n_{t, z} :=
\begin{cases}
C^n_{t, z} \cup \partial C^n_{t, z} & \text{if } C^n_{t, z} \not= \emptyset \text{,} \\
\{z\} & \text{if } C^n_{t, z} = \emptyset \text{.}
\end{cases}
\end{align*}
Finally, we denote by $C_x^{\operatorname{fin}}$ the (countable) set of all finite connected subsets of $T$ which contain the site $x$.

Since $R_{\tau + \epsilon, x}$ (and hence $\overline{R}_{\tau + \epsilon, x}$) is a.s.\ finite by Proposition~\ref{prop critical epsilon}, we have the equality
\begin{align*}
\mf{P} \left[ \sup_{\tau + \delta \leq t \leq \tau + \epsilon} \left| \eta^n_{t, x} - \rho_{t, x} \right| = 0 \right]
= \sum_{A \in C_x^{\operatorname{fin}}} \mf{P} \left[ \sup_{\tau + \delta \leq t \leq \tau + \epsilon} \left| \eta^n_{t, x} - \rho_{t, x} \right| = 0, \overline{R}_{\tau + \epsilon, x} = A \right]
\end{align*}
for all $n \in \mb{N}$. So pick $A \in C_x^{\operatorname{fin}}$ and set $\mb{A} := \left\{ \overline{R}_{\tau + \epsilon, x} = A \right\}$. By the dominated convergence theorem, (\ref{cor general convergence eq statement a}) holds once we know
\begin{align} \label{cor general convergence eq statement b}
\lim_{n \rightarrow \infty} \mf{P} \left[ \left. \sup_{\tau + \delta \leq t \leq \tau + \epsilon} \left| \eta^n_{t, x} - \rho_{t, x} \right| = 0 \right| \mb{A} \right] = 1 \text{.}
\end{align}
It is thus enough to show (\ref{cor general convergence eq statement b}).

Given the set $A$, by Theorem~\ref{thm general convergence} we can choose a sequence $(\alpha(n))_{n \in \mb{N}}$ with $\alpha(n) > 0$ and $\lim_{n \rightarrow \infty} \alpha(n) = 0$ such that the event
\begin{align*}
\mb{C}_n := \Bigl\{ \forall z \in O_{\tau} \cap A \, \exists t \in (\tau - \alpha(n), \tau + \alpha(n)): \eta^n_{t^-, z} > \eta^n_{t, z}, \\
\forall z \in A: G_{\tau - \alpha(n), z} = G_{\tau + \alpha(n), z}, I^n_{\tau + \alpha(n), z} = 0 \Bigr\}
\end{align*}
(where $O_{\tau}$ is defined as in Definition~\ref{defn self-destructive percolation process} and $n \in \mb{N}$ is assumed to be large enough to ensure $A \subset B_n$) satisfies $\lim_{n \rightarrow \infty} \mf{P}[\mb{C}_n] = 1$. As an auxiliary step towards (\ref{cor general convergence eq statement b}), we prove that for all $n \in \mb{N}$ with $A \subset B_n$ the inclusion
\begin{align} \label{cor general convergence eq equality at delta(n)}
\mb{A} \cap \mb{C}_n
\subset \left\{ \forall z \in A: \eta^n_{\tau + \alpha(n), z} = \rho_{\tau + \alpha(n), z} \right\}
\end{align}
holds. So let $z \in A$, let $n \in \mb{N}$ be large enough to ensure $A \subset B_n$ and suppose that the event $\mb{A} \cap \mb{C}_n$ occurs. We distinguish two cases: \\
\emph{Case 1:} $z \in O_{\tau}$. Then there exists $t \in (\tau - \alpha(n), \tau + \alpha(n))$ such that $\eta^n_{t, z} = 0$ holds. Since $G_{\tau - \alpha(n), z} = G_{\tau + \alpha(n), z}$, it follows that we also have $\eta^n_{\tau + \alpha(n), z} = 0$. On the other hand, the assumption $z \in O_{\tau}$ implies $\rho_{\tau, z} = 0$, and from $G_{\tau - \alpha(n), z} = G_{\tau + \alpha(n), z}$ we again deduce $\rho_{\tau + \alpha(n), z} = 0$. Hence we conclude $\eta^n_{\tau + \alpha(n), z} = 0 = \rho_{\tau + \alpha(n), z}$. \\
\emph{Case 2:} $z \not\in O_{\tau}$. By construction $z \not\in O_{\tau}$ implies $\overline{R}_{t, z} \subset \overline{R}_{\tau + \epsilon, z}$ for all $t \in [0, \tau + \epsilon]$. Since we assume $\overline{R}_{\tau + \epsilon, x} = A$ and $z \in A$, we also have $\overline{R}_{\tau + \epsilon, z} \subset A$. In particular we see that $\overline{R}_{\tau - \alpha(n), z} \subset A$ holds. Together with the fact that $I^n_{\tau - \alpha(n), w} = 0$ for all $w \in A$ this yields $\overline{C}^n_{\tau - \alpha(n), z} = \overline{R}_{\tau - \alpha(n), z} \subset A$. If we now use that $G_{\tau - \alpha(n), w} = G_{\tau + \alpha(n), w}$ and $I^n_{\tau - \alpha(n), w} = I^n_{\tau + \alpha(n), w}$ hold for all $w \in A$, it follows that $\overline{C}^n_{\tau + \alpha(n), z} = \overline{R}_{\tau + \alpha(n), z}$, which shows $\eta^n_{\tau + \alpha(n), z} = \rho_{\tau + \alpha(n), z}$.

Having proved (\ref{cor general convergence eq equality at delta(n)}), we now observe that the event
\begin{align*}
\mb{D}_n := \left\{ \forall z \in A: I^n_{\tau + \alpha(n), z} = I^n_{\tau + \epsilon, z} \right\}
\end{align*}
also satisfies $\lim_{n \rightarrow \infty} \mf{P}[\mb{D}_n] = 1$ and that
\begin{align} \label{cor general convergence eq equality at epsilon}
\mb{A} \cap \mb{D}_n \cap \left\{ \forall z \in A: \eta^n_{\tau + \alpha(n), z} = \rho_{\tau + \alpha(n), z} \right\}
\subset \left\{ \sup_{z \in A, \tau + \alpha(n) \leq t \leq \tau + \epsilon} \left| \eta^n_{t, z} - \rho_{t, z} \right| = 0 \right\}
\end{align}
holds for all $n \in \mb{N}$ with $A \subset B_n$. Since we have $\lim_{n \rightarrow \infty} \mf{P} \left[ \left. \mb{C}_n \cap \mb{D}_n \right| \mb{A} \right] = 1$ and $\alpha(n) < \delta$ for $n$ large enough, equation (\ref{cor general convergence eq statement b}) follows from (\ref{cor general convergence eq equality at delta(n)}) and (\ref{cor general convergence eq equality at epsilon}).
\end{proof}

\section{Proof of Theorem \ref{thm general convergence}} \label{sec proof general convergence}

We first prove some general properties of the pure growth process in Section \ref{subsec properties pure growth process} before we come to the core of the proof of Theorem \ref{thm general convergence} in Section \ref{subsec proof general convergence}.

\subsection{Properties of the pure growth process} \label{subsec properties pure growth process}

Let $(\sigma_{t, z}, G_{t, z})_{t \geq 0, z \in T}$ be a pure growth process on $T$ under some probability measure $\mf{P}$. For $x \in T$, $t \geq 0$ and $n \in \mb{N}_0$, let $S_{t, x}$ denote the cluster of $x$ in the configuration $(\sigma_{t, z})_{z \in T}$ and let
\begin{align*}
S^n_{t, x} := S_{t, x} \cap B_n
\end{align*}
be the set of vertices in $S_{t, x}$ whose graph distance from the root $\emptyset$ is at most $n$. Recall the definition of $m(t)$ and $\sigma^2(t)$ in equations (\ref{eq mean binomial dist}) and (\ref{eq variance binomial dist}). We start with some estimates for the first and second moment of $|S^n_{t, \emptyset}|$ in the supercritical case $t > t_c$:

\begin{lem} \label{lem Galton Watson moments}
Let $t > t_c$ and $n \in \mb{N}_0$. Then we have
\begin{align}
1 - e^{-t} \leq \mf{E}_{\mf{P}} \left[ \frac{|S^n_{t, \emptyset}|}{m(t)^n} \right] &\leq \frac{m(t)}{m(t) - 1} \text{,} \label{lem Galton Watson moments eq 1st moment} \\
\mf{E}_{\mf{P}} \left[ \frac{|S^n_{t, \emptyset}|^2}{m(t)^{2n}} \right] &\leq \left( \frac{\sigma^2(t)}{m(t)(m(t) - 1)} + 1 \right) \left( \frac{m(t)}{m(t) - 1} \right)^2 \text{.} \label{lem Galton Watson moments eq 2nd moment}
\end{align}
\end{lem}

\begin{proof}
Let $t > t_c$ and abbreviate $m := m(t)$, $\sigma^2 := \sigma^2(t)$. We will prove (\ref{lem Galton Watson moments eq 1st moment}) and (\ref{lem Galton Watson moments eq 2nd moment}) by means of Galton-Watson theory. So let $X_{n, i}$, $n,i \in \mb{N}$, be i.i.d.\ $\{0, 1, \ldots, r\}$-valued random variables under some probability measure $\mf{\tilde P}$ such that $X_{n, i}$ is binomially distributed with parameters $r$ and $1 - e^{-t}$. (In particular, $X_{n, i}$ has mean $m$ and variance $\sigma^2$.) Define $Z_n$, $n \in \mb{N}_0$, recursively by $Z_0 := 1$ and $Z_n := \sum_{i = 1}^{Z_{n - 1}} X_{n, i}$, $n \in \mb{N}$, and set $S_n := \sum_{i = 0}^n Z_i$, $n \in \mb{N}_0$. Then $Z_n$, $n \in \mb{N}_0$, is a supercritical Galton-Watson process, and $Z_n$ has mean
\begin{align}
\mf{E}_{\mf{\tilde P}} \left[ Z_n \right]
&= m^n \label{lem Galton Watson moments eq offspring mean}
\intertext{and variance}
\mf{Var}_{\mf{\tilde P}} \left[ Z_n \right]
&= \sigma^2 m^{n - 1} \frac{m^n - 1}{m - 1} \label{lem Galton Watson moments eq offspring variance}
\end{align}
(see e.g.\ \cite{harris1963branching}, Section I.5). Moreover, let $U$ be a $\{0, 1\}$-valued random variable on the same probability space which is independent from $X_{n, i}$, $n,i \in \mb{N}$, and Bernoulli distributed with parameter $1 - e^{-t}$. Then the distribution of $|S^n_{t, \emptyset}|$ under $\mf{P}$ and the distribution of $US_n$ under $\mf{\tilde P}$ coincide, and $\mf{E}_{\mf{\tilde P}}[U] = 1 - e^{-t} \leq 1$. For the proof of (\ref{lem Galton Watson moments eq 1st moment}) and (\ref{lem Galton Watson moments eq 2nd moment}), it therefore suffices to show the following inequalities for $n \in \mb{N}_0$:
\begin{align}
m^n \leq \mf{E}_{\mf{\tilde P}} \left[ S_n \right]
&\leq \frac{m}{m - 1} m^n \text{,} \label{lem Galton Watson moments eq 1st moment alt} \\
\mf{E}_{\mf{\tilde P}} \left[ S_n^2 \right]
&\leq \left( \frac{\sigma^2}{m(m - 1)} + 1 \right) \left( \frac{m}{m - 1} \right)^2 m^{2n} \text{.} \label{lem Galton Watson moments eq 2nd moment alt}
\end{align}

\emph{Proof of (\ref{lem Galton Watson moments eq 1st moment alt}):} Using equation (\ref{lem Galton Watson moments eq offspring mean}), we obtain
\begin{align*}
\mf{E}_{\mf{\tilde P}} \left[ S_n \right]
= \sum_{i = 0}^n m^i
\leq \frac{m}{m - 1} m^n
\end{align*}
for all $n \in \mb{N}_0$, which proves both sides of (\ref{lem Galton Watson moments eq 1st moment alt}).

\emph{Proof of (\ref{lem Galton Watson moments eq 2nd moment alt}):} For $i \in \mb{N}_0$, we easily deduce from equations (\ref{lem Galton Watson moments eq offspring mean}) and (\ref{lem Galton Watson moments eq offspring variance})
\begin{align*}
\mf{E}_{\mf{\tilde P}} \left[ Z_i^2 \right]
= \sigma^2 m^{i - 1} \frac{m^i - 1}{m - 1} + m^{2i}
\leq \left( \frac{\sigma^2}{m(m - 1)} + 1 \right) m^{2i} \text{.}
\end{align*}
Furthermore, for $i,j \in \mb{N}_0$ with $i < j$, we have
\begin{align*}
\mf{E}_{\mf{\tilde P}} \left[ Z_i Z_j \right]
= \mf{E}_{\mf{\tilde P}} \left[ Z_i^2 \right] m^{j - i}
\leq \left( \frac{\sigma^2}{m(m - 1)} + 1 \right) m^{i + j} \text{.}
\end{align*}
We thus obtain
\begin{align*}
\mf{E}_{\mf{\tilde P}} \left[ S_n^2 \right]
= \sum_{i, j = 0}^n \mf{E}_{\mf{\tilde P}} \left[ Z_i Z_j \right]
\leq \left( \frac{\sigma^2}{m(m - 1)} + 1 \right) \sum_{i, j = 0}^n m^{i + j}
\end{align*}
for all $n \in \mb{N}_0$. The last sum can be bounded from above by
\begin{align*}
\sum_{i, j = 0}^n m^{i + j}
= \left( \sum_{i = 0}^n m^i \right)^2
\leq \left( \frac{m}{m - 1} \right)^2 m^{2n} \text{,}
\end{align*}
which completes the proof of (\ref{lem Galton Watson moments eq 2nd moment alt}).
\end{proof}

Recall equations (\ref{eq Kesten-Stigum 1}) and (\ref{eq Kesten-Stigum 2}). We now want to prove similar statements which are uniform in $t$. The price we pay for this kind of uniformity is that in contrast to (\ref{eq Kesten-Stigum 1}) and (\ref{eq Kesten-Stigum 2}), our statements are in probability rather than almost surely. The precise formulation is as follows:

\begin{prop} \label{prop replacement Kesten-Stigum}
Let $x \in T$ and $a > t_c$. Then we have
\begin{align} \label{prop replacement Kesten-Stigum eq upper bound}
\lim_{C \rightarrow \infty} \, \sup_{n \in \mb{N}:\, n > |x|} \, \sup_{t \in [a, \infty)} \mf{P} \left[ |S^n_{t, x}| > Cm(t)^n \right] = 0
\end{align}
and
\begin{align} \label{prop replacement Kesten-Stigum eq lower bound}
\lim_{c \downarrow 0} \, \sup_{n \in \mb{N}:\, n > |x|} \, \sup_{t \in [a, \infty)} \mf{P} \left[ |S^n_{t, x}| < cm(t)^n \left| |S_{t, x}| = \infty \right. \right] = 0 \text{.}
\end{align}
\end{prop}

\begin{proof}
Let $a > t_c$.

\emph{Step 1:} We first prove (\ref{prop replacement Kesten-Stigum eq upper bound}) and (\ref{prop replacement Kesten-Stigum eq lower bound}) for $x = \emptyset$.

For $C > 0$ and $n \in \mb{N}$, $t \in [a, \infty)$, the Markov inequality and equation (\ref{lem Galton Watson moments eq 1st moment}) yield
\begin{align} \label{prop replacement Kesten-Stigum eq Markov}
\mf{P} \left[ \frac{|S^n_{t, \emptyset}|}{m(t)^n} \geq C \right]
\leq \frac{1}{C} \, \mf{E}_{\mf{P}} \left[ \frac{|S^n_{t, \emptyset}|}{m(t)^n} \right]
\leq \frac{1}{C} \frac{m(t)}{m(t) - 1} \text{.}
\end{align}
Since $m(t)$ is bounded away from $1$ for $t \in [a, \infty)$, this implies (\ref{prop replacement Kesten-Stigum eq upper bound}) for $x = \emptyset$.

As preparatory work for the proof of (\ref{prop replacement Kesten-Stigum eq lower bound}) we next show that there exist $c, \delta > 0$ such that
\begin{align} \label{prop replacement Kesten-Stigum eq weak lower bound}
\sup_{n \in \mb{N}} \sup_{t \in [a, \infty)} \mf{P} \left[ \frac{|S^n_{t, \emptyset}|}{m(t)^n} \geq c \right] \geq \delta
\end{align}
holds. For arbitrary $0 < c < C$ and $n \in \mb{N}$, $t \in [a, \infty)$, we have
\begin{align*}
1 - e^{-t}
&\leq \mf{E}_{\mf{P}} \left[ \frac{|S^n_{t, \emptyset}|}{m(t)^n} \right] \\
&\leq c + C \, \mf{P} \left[ \frac{|S^n_{t, \emptyset}|}{m(t)^n} \geq c \right] + \mf{E}_{\mf{P}} \left[ \frac{|S^n_{t, \emptyset}|}{m(t)^n} \, 1_{\{\frac{|S^n_{t, \emptyset}|}{m(t)^n} \geq C\}} \right] \text{,}
\end{align*}
where the first inequality is due to (\ref{lem Galton Watson moments eq 1st moment}) and the second inequality is obtained by distinguishing in which of the intervals $[0, c)$, $[c, C)$, $[C, \infty)$ the rescaled cluster size $|S^n_{t, \emptyset}|/m(t)^n$ lies. The last summand can be bounded from above by
\begin{align*}
\mf{E}_{\mf{P}} \left[ \frac{|S^n_{t, \emptyset}|}{m(t)^n} \, 1_{\{\frac{|S^n_{t, \emptyset}|}{m(t)^n} \geq C\}} \right]
&\leq \left( \mf{E}_{\mf{P}} \left[ \frac{|S^n_{t, \emptyset}|^2}{m(t)^{2n}} \right] \right)^{1/2} \left( \mf{P} \left[ \frac{|S^n_{t, \emptyset}|}{m(t)^n} \geq C \right] \right)^{1/2} \\
&\leq \frac{1}{C^{1/2}} \left( \frac{\sigma^2(t)}{m(t)(m(t) - 1)} + 1 \right)^{1/2} \left( \frac{m(t)}{m(t) - 1} \right)^{3/2} \text{,}
\end{align*}
where we first use the Cauchy-Schwarz inequality and then apply equations (\ref{lem Galton Watson moments eq 2nd moment}) and (\ref{prop replacement Kesten-Stigum eq Markov}). We thus obtain
\begin{align*}
\mf{P} \left[ \frac{|S^n_{t, \emptyset}|}{m(t)^n} \geq c \right]
\geq \frac{1}{C} \left( 1 - e^{-t} - c - \frac{1}{C^{1/2}} \left( \frac{\sigma^2(t)}{m(t)(m(t) - 1)} + 1 \right)^{1/2} \left( \frac{m(t)}{m(t) - 1} \right)^{3/2} \right) \text{.}
\end{align*}
Since $\sigma^2(t)/m(t) = e^{-t}$ and $m(t)$ is bounded away from $1$ for $t \in [a, \infty)$, this proves the existence of $c, \delta > 0$ satisfying (\ref{prop replacement Kesten-Stigum eq weak lower bound}).

We now prove (\ref{prop replacement Kesten-Stigum eq lower bound}). Intuitively, (\ref{prop replacement Kesten-Stigum eq lower bound}) follows from (\ref{prop replacement Kesten-Stigum eq weak lower bound}) because conditionally on $\{|S_{t, \emptyset}| = \infty\}$, the cluster $S_{t, \emptyset}$ contains arbitrarily many independent subtrees in which an asymptotic growth of the form (\ref{prop replacement Kesten-Stigum eq weak lower bound}) can occur. The formal proof goes as follows:

Let $\epsilon > 0$. We first construct a finite set $U \subset [a, \infty)$ such that
\begin{align} \label{prop replacement Kesten-Stigum eq set U}
\forall t \in [a, \infty) \, \exists u \in U: u \leq t,
\mf{P} \left[ |S_{u, \emptyset}| = \infty \right] \geq \mf{P} \left[ |S_{t, \emptyset}| = \infty \right] - \frac{\epsilon}{4}
\end{align}
holds: Define $f: [a, \infty) \rightarrow [0, 1], f(t) := \mf{P} \left[ |S_{t, \emptyset}| = \infty \right]$, and
\begin{align*}
R := \left\{ f(a) + i \frac{\epsilon}{4}: i \in \mb{N}_0 \right\} \cap [0, 1) \text{.}
\end{align*}
Then $R$ is clearly finite. Since $f$ is continuous, strictly monotone increasing and maps $[a, \infty)$ onto $[f(a), 1)$, it follows that $U := f^{-1}(R)$ is finite and satisfies (\ref{prop replacement Kesten-Stigum eq set U}).

Let $\delta, c > 0$ be as in equation (\ref{prop replacement Kesten-Stigum eq weak lower bound}). Given $\epsilon, \delta, c$, we choose constants $k, l \in \mb{N}$ in the following way: First, we take $k \in \mb{N}$ such that $(1 - \delta)^k \leq \epsilon/4$ holds. Then we choose $l \in \mb{N}$ such that
\begin{align} \label{prop replacement Kesten-Stigum eq symmetric difference}
\forall u \in U: \mf{P} \left[ \left\{ |S_{u, \emptyset} \cap T_l| \geq k \right\} \triangle \left\{ |S_{u, \emptyset}| = \infty \right\} \right] \leq \frac{\epsilon}{4}
\end{align}
holds, where $\triangle$ denotes the symmetric difference: For each individual $u \in U$ such an $l$ exists because of the Kesten-Stigum theorem (whose full statement is of course much stronger), and since the set $U$ is finite, we can choose $l$ uniformly for all $u \in U$. Finally, we set $\tilde c := c/r^l$. Now let $n \in \{l + 1, l + 2, \ldots\}$ and $t \in [a, \infty)$ be arbitrary. Given $t$, choose $u \in U$ as in (\ref{prop replacement Kesten-Stigum eq set U}). Then we can make the estimates
\begin{align}
\mf{P} \left[ \frac{|S^n_{t, \emptyset}|}{m(t)^n} \geq \tilde c, |S_{t, \emptyset}| = \infty \right]
&\geq \mf{P} \left[ \frac{|S^n_{t, \emptyset}|}{m(t)^n} \geq \tilde c, |S_{u, \emptyset}| = \infty \right] \nonumber \\
&\geq \mf{P} \left[ \frac{|S^n_{t, \emptyset}|}{m(t)^n} \geq \tilde c, |S_{u, \emptyset} \cap T_l| \geq k \right] - \frac{\epsilon}{4} \text{,} \label{prop replacement Kesten-Stigum eq large estimate 1}
\end{align}
where the first inequality holds because of $u \leq t$ and the second inequality follows from (\ref{prop replacement Kesten-Stigum eq symmetric difference}). On the event $\left\{ |S_{u, \emptyset} \cap T_l| \geq k \right\}$, let $Z_{u, 1}, \ldots, Z_{u, k}$ be an enumeration of the ``first'' $k$ vertices in $S_{u, \emptyset} \cap T_l$. For $z \in T$ let
\begin{align} \label{prop replacement Kesten-Stigum eq upward cluster 1}
\hat S_{t, z} := \left\{ v \in S_{t, z}: z \preceq v \right\}
\end{align}
be the part of the cluster of $S_{t, z}$ which lies in the $r$-regular rooted subtree of $T$ originating from $z$ and let
\begin{align} \label{prop replacement Kesten-Stigum eq upward cluster 2}
\hat S^n_{t, z} := \hat S_{t, z} \cap B_n \text{.}
\end{align}
Since $u \leq t$, on the event $\left\{ |S_{u, \emptyset} \cap T_l| \geq k \right\}$, we have $Z_{u, i} \in S_{t, \emptyset}$ and hence $\hat S_{t, Z_{u, i}} \subset S_{t, \emptyset}$ for all $i \in \{1, \ldots, k\}$. This gives
\begin{align}
\mf{P} \left[ \frac{|S^n_{t, \emptyset}|}{m(t)^n} \geq \tilde c, |S_{u, \emptyset} \cap T_l| \geq k \right]
&\geq \mf{P} \left[ \exists i \in \{1, \ldots k\}: \frac{|\hat S^n_{t, Z_{u, i}}|}{m(t)^n} \geq \tilde c, |S_{u, \emptyset} \cap T_l| \geq k \right] \nonumber \\
&= \left( 1 - \left( 1 - \mf{P} \left[ \left. \frac{|S^{n - l}_{t, \emptyset}|}{m(t)^n} \geq \tilde c \right| \sigma_{t, \emptyset} = 1 \right] \right)^k \right) \, \mf{P} \left[ |S_{u, \emptyset} \cap T_l| \geq k \right] \text{,} \label{prop replacement Kesten-Stigum eq large estimate 2}
\end{align}
where the last equality follows from the following observations about the pure growth process:
\begin{compactitem}
\item The configuration on $B_l$ at time $u$ and the configuration on $T \setminus B_l$ at time $t$ are independent.
\item The configurations at time $t$ on the $r$-regular rooted subtrees originating from the vertices in $T_l$ are independent and identically distributed as the configuration at time $t$ on the entire tree $T$.
\end{compactitem}
Using the inequality $\tilde c m(t)^l \leq c$ and the defining equations for $c$, $\delta$ and $k$, we can estimate the first factor in (\ref{prop replacement Kesten-Stigum eq large estimate 2}) by
\begin{align}
1 - \left( 1 - \mf{P} \left[ \left. \frac{|S^{n - l}_{t, \emptyset}|}{m(t)^n} \geq \tilde c \right| \sigma_{t, \emptyset} = 1 \right] \right)^k
&\geq 1 - \left( 1 - \mf{P} \left[ \frac{|S^{n - l}_{t, \emptyset}|}{m(t)^n} \geq \tilde c \right] \right)^k \nonumber \\
&\geq 1 - \left( 1 - \mf{P} \left[ \frac{|S^{n - l}_{t, \emptyset}|}{m(t)^{n - l}} \geq c \right] \right)^k \nonumber \\
&\geq 1 - \left( 1 - \delta \right)^k \geq 1 - \frac{\epsilon}{4} \label{prop replacement Kesten-Stigum eq large estimate 3} \text{.}
\end{align}
The second factor in (\ref{prop replacement Kesten-Stigum eq large estimate 2}) is bounded from below by
\begin{align}
\mf{P} \left[ |S_{u, \emptyset} \cap T_l| \geq k \right]
&\geq \mf{P} \left[ |S_{u, \emptyset}| = \infty \right] - \frac{\epsilon}{4} \nonumber \\
&\geq \mf{P} \left[ |S_{t, \emptyset}| = \infty \right] - \frac{\epsilon}{2} \label{prop replacement Kesten-Stigum eq large estimate 4}
\end{align}
because of (\ref{prop replacement Kesten-Stigum eq symmetric difference}) and (\ref{prop replacement Kesten-Stigum eq set U}). Putting equations (\ref{prop replacement Kesten-Stigum eq large estimate 1}), (\ref{prop replacement Kesten-Stigum eq large estimate 2}), (\ref{prop replacement Kesten-Stigum eq large estimate 3}) and (\ref{prop replacement Kesten-Stigum eq large estimate 4}) together, we obtain
\begin{align*}
\mf{P} \left[ \frac{|S^n_{t, \emptyset}|}{m(t)^n} \geq \tilde c, |S_{t, \emptyset}| = \infty \right]
\geq \mf{P} \left[ |S_{t, \emptyset}| = \infty \right] - \epsilon \text{.}
\end{align*}
Since this holds uniformly for $n \in \{l + 1, l + 2, \ldots\}$ and $t \in [a, \infty)$ and since $\mf{P} \left[ |S_{t, \emptyset}| = \infty \right] \geq \mf{P} \left[ |S_{a, \emptyset}| = \infty \right] > 0$ for all $t \in [a, \infty)$, we conclude
\begin{align} \label{prop replacement Kesten-Stigum eq lower bound epsilon}
\sup_{n \in \{l + 1, l + 2, \ldots\}} \sup_{t \in [a, \infty)} \mf{P} \left[ \left. \frac{|S^n_{t, \emptyset}|}{m(t)^n} \geq \tilde c \right| |S_{t, \emptyset}| = \infty \right]
\geq 1 - \frac{\epsilon}{\mf{P} \left[ |S_{a, \emptyset}| = \infty \right]} \text{.}
\end{align}
Additionally, we also have the trivial estimate
\begin{align*}
\sup_{n \in \{1, \ldots, l\}} \sup_{t \in [a, \infty)} \mf{P} \left[ \left. \frac{|S^n_{t, \emptyset}|}{m(t)^n} \geq \frac{1}{r^l} \right| |S_{t, \emptyset}| = \infty \right]
= 1
\end{align*}
for $n \in \{1, \ldots, l\}$. Together with (\ref{prop replacement Kesten-Stigum eq lower bound epsilon}) this proves (\ref{prop replacement Kesten-Stigum eq lower bound}) for $x = \emptyset$.

\emph{Step 2:} We now prove equations (\ref{prop replacement Kesten-Stigum eq upper bound}) and (\ref{prop replacement Kesten-Stigum eq lower bound}) for general $x \in T$. So let $x \in T$ and let $n \in \{|x| + 1, |x| + 2, \ldots\}$, $t \in [a, \infty)$. For both equations we distinguish which vertex $z$ of the finitely many ancestors of $x$ is the root of the cluster $S_{t, x}$ (the case $S_{t, x} = \emptyset$ being irrelevant) and then use the fact that the $r$-regular rooted subtree originating from $z$ is isomorphic to $T$. Let $\hat S_{t, z}$ and $\hat S^n_{t, z}$ be defined as in (\ref{prop replacement Kesten-Stigum eq upward cluster 1}) and (\ref{prop replacement Kesten-Stigum eq upward cluster 2}) respectively. Regarding (\ref{prop replacement Kesten-Stigum eq upper bound}) we then obtain for all $C > 0$
\begin{align*}
\mf{P} \left[ |S^n_{t, x}| > Cm(t)^n \right]
&= \sum_{z \in T:\, z \preceq x} \mf{P} \left[ |S^n_{t, x}| > Cm(t)^n, z \text{ is the root of } S_{t, x} \right] \\
&\leq \sum_{z \in T:\, z \preceq x} \mf{P} \left[ |\hat S^n_{t, z}| > Cm(t)^n \right] \\
&= \sum_{z \in T:\, z \preceq x} \mf{P} \left[ |S^{n - |z|}_{t, \emptyset}| > Cm(t)^n \right] \\
&\leq \sum_{z \in T:\, z \preceq x} \mf{P} \left[ |S^{n - |z|}_{t, \emptyset}| > Cm(a)^{|z|} \cdot m(t)^{n - |z|} \right] \text{,}
\end{align*}
and regarding (\ref{prop replacement Kesten-Stigum eq lower bound}) we similarly obtain for all $c > 0$
\begin{align*}
\mf{P} \left[ \left. |S^n_{t, x}| < cm(t)^n \right| |S_{t, x}| = \infty \right]
&= \sum_{z \in T:\, z \preceq x} \frac{\mf{P} \left[ |S^n_{t, x}| < cm(t)^n, z \text{ is the root of } S_{t, x}, |S_{t, x}| = \infty \right]}{\mf{P} \left[ |S_{t, x}| = \infty \right] } \\
&\leq \sum_{z \in T:\, z \preceq x} \frac{\mf{P} \left[ |\hat S^n_{t, z}| < cm(t)^n, |\hat S_{t, z}| = \infty \right]}{\mf{P} \left[ |S_{t, \emptyset}| = \infty \right] } \\
&= \sum_{z \in T:\, z \preceq x} \frac{\mf{P} \left[ |S^{n - |z|}_{t, \emptyset}| < cm(t)^n, |S_{t, \emptyset}| = \infty \right]}{\mf{P} \left[ |S_{t, \emptyset}| = \infty \right] } \\
&\leq \sum_{z \in T:\, z \preceq x} \mf{P} \left[ \left. |S^{n - |z|}_{t, \emptyset}| < cr^{|z|} \cdot m(t)^{n - |z|} \right| |S_{t, \emptyset}| = \infty \right] \text{.}
\end{align*}
Together with Step 1 this completes the proof of (\ref{prop replacement Kesten-Stigum eq upper bound}) and (\ref{prop replacement Kesten-Stigum eq lower bound}).
\end{proof}

\subsection{The core of the proof of Theorem~\ref{thm general convergence}} \label{subsec proof general convergence}

Throughout this section, consider the setup of Theorem~\ref{thm general convergence}: Let $\tau \in (t_c, \infty)$ and suppose that $\lambda: \mb{N} \rightarrow (0, \infty)$ satisfies $\lambda(n) \approx 1/m(\tau)^n$ for $n \rightarrow \infty$. For $n \in \mb{N}$, let $(\eta^n_{t, z}, G_{t, z}, I^n_{t, z})_{t \geq 0, z \in B_n}$ be a forest-fire process on $B_n$ with parameter $\lambda(n)$ and let $(\sigma_{t, z}, G_{t, z})_{t \geq 0, z \in T}$ be a pure growth process on $T$, coupled in the canonical way under some probability measure $\mf{P}$. As before, we use the following notation: For $x \in T$, $t \geq 0$ and $n \in \mb{N}$, let $S_{t, x}$ denote the cluster of $x$ in the configuration $(\sigma_{t, z})_{z \in T}$ of the pure growth process at time~$t$, let $S^n_{t, x} := S_{t, x} \cap B_n$ and let $O_t := \left\{ z \in T: |S_{t, z}| = \infty \right\}$. Similarly, let $C^n_{t, x}$ denote the cluster of $x$ in the configuration $(\eta^n_{t, z})_{z \in B_n}$ of the forest-fire process at time~$t$.

Choose an arbitrary function $f: \mb{N} \rightarrow (0, \infty)$ which satisfies
\begin{align} \label{eq condition on f}
1 \ll f(n) \ll \frac{n}{\log n} \qquad \text{for } n \rightarrow \infty \text{.}
\end{align}
Define a corresponding sequence $(\tau_n)_{n \in \mb{N}}$ of time points in such a way that $\lambda(n) = f(n)/m(\tau_n)^n$ holds, i.e.\
\begin{align*}
\tau_n := m^{-1} \left( \sqrt[n]{\frac{f(n)}{\lambda(n)}} \right) \text{,}
\end{align*}
where
\begin{align*}
m^{-1}(y) = \log \frac{r}{r - y} \text{,} \qquad y \in [0, r) \text{,}
\end{align*}
denotes the inverse function of $m(t)$, $t \geq 0$. Since $\sqrt[n]{\lambda(n)} \rightarrow 1/m(\tau)$ and $\sqrt[n]{f(n)} \rightarrow 1$ for $n \rightarrow \infty$ (the first limit follows from $\lambda(n) \approx 1/m(\tau)^n$ for $n \rightarrow \infty$, the second limit is a consequence of (\ref{eq condition on f})), we then have
\begin{align} \label{eq tau_n to tau}
\lim_{n \rightarrow \infty} \tau_n = \tau \text{.}
\end{align}
In particular, for all $x \in T$ it is true that
\begin{align} \label{eq S_tau_n to S_tau}
\lim_{n \rightarrow \infty} \mf{P} \left[ |S_{\tau_n, x}| = |S_{\tau, x}| \right] = 1 \text{.}
\end{align}
Equations (\ref{eq tau_n to tau}) and (\ref{eq S_tau_n to S_tau}) imply that for the proof of Theorem~\ref{thm general convergence}, it is enough to verify the following statement: For all finite subsets $E \subset T$ and for all $\delta > 0$
\begin{align} \label{eq thm general convergence with tau_n}
\lim_{n \rightarrow \infty} \mf{P} \left[ \sup_{z \in E, 0 \leq t \leq \tau_n} \left| \eta^n_{t, z} - \sigma_{t, z} \right| = 0, \forall z \in O_{\tau_n} \cap E \, \exists t \in (\tau_n, \tau_n + \delta): \eta^n_{t^-, z} > \eta^n_{t, z} \right] = 1
\end{align}
holds.

Since $E$ is finite, we may assume without loss of generality that $E$ is a singleton, i.e.\ $E = \{x\}$ for some $x \in T$. So let $x \in T$ and $\delta > 0$ be fixed and define
\begin{align*}
\mf{Q}_n \left[ \,\cdot\, \right]
:= \mf{P} \left[ \,\cdot\, \left| |S_{\tau_n, x}| = \infty \right. \right]
\end{align*}
for all $n \in \mb{N}$ which satisfy $\tau_n > t_c$. (Due to (\ref{eq tau_n to tau}) the case $\tau_n \leq t_c$ can only occur for finitely many $n$.) It then suffices to prove
\begin{align} \label{eq thm general convergence part 1}
\lim_{n \rightarrow \infty} \mf{P} \left[ \sup_{0 \leq t \leq \tau_n} \left| \eta^n_{t, x} - \sigma_{t, x} \right| = 0 \right] = 1
\end{align}
and
\begin{align} \label{eq thm general convergence part 2}
\lim_{n \rightarrow \infty} \mf{Q}_n \left[ \exists t \in (\tau_n, \tau_n + \delta): \eta^n_{t^-, x} > \eta^n_{t, x} = 0 \right] = 1 \text{.}
\end{align}

Before we go into the details, let us briefly outline the strategy for the proof of (\ref{eq thm general convergence part 1}) and (\ref{eq thm general convergence part 2}): We investigate how the vertices in the cluster $S^n_{\tau_n, x}$ of the pure growth process on $B_n$ at time $\tau_n$ behave in the forest-fire process on $B_n$ between time $0$ and time $\tau_n$. We will see that typically destruction only occurs in high generations of $S^n_{\tau_n, x}$ and only few vertices in $S^n_{\tau_n, x}$ are affected by destruction. This has two consequences: Firstly, it shows that (\ref{eq thm general convergence part 1}) holds indeed. Secondly, it implies that if $S_{\tau_n, x}$ is infinite, then the cluster $C^n_{\tau_n, x}$ has the same order of magnitude as $S^n_{\tau_n, x}$, namely $m(\tau_n)^n$. But since $\lambda(n) m(\tau_n)^n = f(n)$ and $f(n) \rightarrow \infty$ for $n \rightarrow \infty$, it follows that $C^n_{\tau_n, x}$ is typically hit by ignition soon after time $\tau_n$, which proves (\ref{eq thm general convergence part 2}).

We now make these arguments rigorous. In doing so, we will use the following Landau-type notation: If $X^n$, $n \in \mb{N}$, is a sequence of real-valued random variables under the probability measure $\mf{P}$ and $h: \mb{N} \rightarrow [0, \infty)$ is a non-negative function, we write
\begin{align*}
X^n \stackrel{\mf{P}}{=} O(h(n)) \text{ for } n \rightarrow \infty &:\Leftrightarrow \lim_{c \rightarrow \infty} \liminf_{n \rightarrow \infty} \mf{P} \left[ |X^n| \leq ch(n) \right] = 1 \text{;} \\
X^n \stackrel{\mf{P}}{=} \Omega(h(n)) \text{ for } n \rightarrow \infty &:\Leftrightarrow \lim_{c \rightarrow \infty} \liminf_{n \rightarrow \infty} \mf{P} \left[ |X^n| \geq \frac{1}{c} h(n) \right] = 1 \text{.}
\end{align*}

\begin{lem} \label{lem distance lightning}
Let
\begin{align*}
\iota^n := \inf \left\{ t \in [0, \tau_n): \left( \exists z \in S^n_{\tau_n, x}: I^n_{(\tau_n - t)^-, z} < I^n_{\tau_n - t, z} \right) \right\} \wedge \tau_n
\end{align*}
be the amount of time between $\tau_n$ and the last time of lightning in $S^n_{\tau_n, x}$ before $\tau_n$. (On the event $\{ \forall z \in S^n_{\tau_n, x}: I^n_{\tau_n, z} = 0 \}$ we have $\iota^n = \tau_n$ by definition.) Then we have
\begin{align*}
\iota^n \stackrel{\mf{P}}{=} \Omega \left( \frac{1}{f(n)} \right) \qquad \text{for } n \rightarrow \infty \text{.}
\end{align*}
\end{lem}

\begin{proof}
Let $c, \tilde c > 0$, $n \in \mb{N}$ with $\tau_n \geq 1/(cf(n))$ and let
\begin{align} \label{lem distance lightning eq E}
\mb{E}_{n, \tilde c} := \left\{ |S^n_{\tau_n, x}| \leq \tilde c m(\tau_n)^n \right\} \text{.}
\end{align}
By Proposition \ref{prop replacement Kesten-Stigum}, equation (\ref{prop replacement Kesten-Stigum eq upper bound}), it suffices to show
\begin{align} \label{lem distance lightning eq essence}
\forall \tilde  c > 0: \lim_{c \rightarrow \infty} \limsup_{n \rightarrow \infty} \mf{P} \left[ \iota^n < \frac{1}{cf(n)}, \mb{E}_{n, \tilde c} \right] = 0 \text{.}
\end{align}

Indeed, we have
\begin{align*}
\mf{P} \left[ \iota^n < \frac{1}{cf(n)}, \mb{E}_{n, \tilde c} \right]
&= \mf{E}_{\mf{P}} \left[ \mf{P} \left[ \left. \iota^n < \frac{1}{cf(n)} \right| S^n_{\tau_n, x} \right] 1_{\mb{E}_{n, \tilde c}} \right] \\
&= \mf{E}_{\mf{P}} \left[ \left( 1 - \exp \left( -\frac{1}{cf(n)} \lambda(n) |S^n_{\tau_n, x}| \right) \right) 1_{\mb{E}_{n, \tilde c}} \right] \\
&\leq 1 - \exp \left( -\frac{\tilde c}{c} \right)
\xrightarrow[c \rightarrow \infty]{} 0 \text{,}
\end{align*}
which proves (\ref{lem distance lightning eq essence}).
\end{proof}

\begin{lem} \label{lem number lightning}
Let
\begin{align*}
N^n := \sum_{z \in S^n_{\tau_n, x}} I^n_{\tau_n, z}
\end{align*}
be the number of lightnings in $S^n_{\tau_n, x}$ up to time $\tau_n$. Then we have
\begin{align*}
N^n \stackrel{\mf{P}}{=} \mc{O} \left( f(n) \right) \qquad \text{for } n \rightarrow \infty \text{.}
\end{align*}
\end{lem}

\begin{proof}
Let $c, \tilde c > 0$, $n \in \mb{N}$ and let $\mb{E}_{n, \tilde c}$ be defined as in (\ref{lem distance lightning eq E}). By Proposition \ref{prop replacement Kesten-Stigum}, equation (\ref{prop replacement Kesten-Stigum eq upper bound}), it suffices to show
\begin{align} \label{lem number lightning eq essence}
\forall \tilde  c > 0: \lim_{c \rightarrow \infty} \limsup_{n \rightarrow \infty} \mf{P} \left[ N^n > cf(n), \mb{E}_{n, \tilde c} \right] = 0 \text{.}
\end{align}

Indeed, we have
\begin{align*}
\mf{P} \left[ N^n > cf(n), \mb{E}_{n, \tilde c} \right]
&\leq \frac{1}{cf(n)} \, \mf{E}_{\mf{P}} \left[ N^n 1_{\mb{E}_{n, \tilde c}} \right] \\
&= \frac{1}{cf(n)} \, \mf{E}_{\mf{P}} \left[ \mf{E}_{\mf{P}} \left[ N^n \left| S^n_{\tau_n, x} \right. \right] 1_{\mb{E}_{n, \tilde c}} \right] \\
&= \frac{1}{cf(n)} \, \mf{E}_{\mf{Q}_n} \left[ \tau_n \lambda(n) |S^n_{\tau_n, x}| 1_{\mb{E}_{n, \tilde c}} \right] \\
&\leq \frac{\tau_n \tilde c}{c}
\end{align*}
and (by (\ref{eq tau_n to tau}))
\begin{align*}
\lim_{c \rightarrow \infty} \lim_{n \rightarrow \infty} \frac{\tau_n \tilde c}{c}
= \lim_{c \rightarrow \infty} \frac{\tau \tilde c}{c}
= 0 \text{,}
\end{align*}
which proves (\ref{lem number lightning eq essence}).
\end{proof}

\begin{lem} \label{lem depth lightning}
Let
\begin{align*}
K^n := \max \left\{ k \in \{0, \ldots, n\}: \left( \exists z \in S^n_{\tau_n, x} \cap T_{n - k}: I^n_{\tau_n, z} > 0 \right) \right\} \vee (-1)
\end{align*}
be the ``depth'' of lightning in $S^n_{\tau_n, x}$ up to time $\tau_n$. (On the event $\{ \forall z \in S^n_{\tau_n, x}: I^n_{\tau_n, z} = 0 \}$ we have $K^n = -1$ by definition.) Then we have
\begin{align*}
K^n \stackrel{\mf{P}}{=} \mc{O} \left( \log n \right) \qquad \text{for } n \rightarrow \infty \text{.}
\end{align*}
\end{lem}

\begin{proof}
Let $c, \tilde c > 0$, $n \in \mb{N}$ with $n \geq \lfloor c \log n \rfloor + 2$ and let
\begin{align*}
\mb{E}_{n, c, \tilde c} := \left\{ |S^{n - \lfloor c \log n \rfloor - 1}_{\tau_n, x}| \leq \tilde c m(\tau_n)^{n - \lfloor c \log n \rfloor - 1} \right\} \text{.}
\end{align*}
By Proposition \ref{prop replacement Kesten-Stigum}, equation (\ref{prop replacement Kesten-Stigum eq upper bound}), it suffices to show
\begin{align} \label{lem depth lightning eq essence}
\forall \tilde  c > 0: \lim_{c \rightarrow \infty} \limsup_{n \rightarrow \infty} \mf{P} \left[ K^n > c \log n, \mb{E}_{n, c, \tilde c} \right] = 0 \text{.}
\end{align}

Indeed, we have
\begin{align*}
\mf{P} \left[ K^n > c \log n, \mb{E}_{n, c, \tilde c} \right]
&= \mf{E}_{\mf{P}} \left[ \mf{P} \left[ \left. \exists z \in S^{n - \lfloor c \log n \rfloor - 1}_{\tau_n, x}: I^n_{\tau_n, z} > 0 \right| S^{n - \lfloor c \log n \rfloor - 1}_{\tau_n, x} \right] 1_{\mb{E}_{n, c, \tilde c}} \right] \\
&= \mf{E}_{\mf{P}} \left[ \left( 1 - \exp \left( -\tau_n \lambda(n) |S^{n - \lfloor c \log n \rfloor - 1}_{\tau_n, x}| \right) \right) 1_{\mb{E}_{n, c, \tilde c}} \right] \\
&\leq 1 - \exp \left( -\tilde c \tau_n \frac{f(n)}{m(\tau_n)^{\lfloor c \log n \rfloor + 1}} \right) \text{.}
\end{align*}
Equations (\ref{eq condition on f}) and (\ref{eq tau_n to tau}) imply that for $n$ large enough $f(n) \leq n$ and $m(\tau_n) > 1$ hold and hence
\begin{align*}
\tau_n \frac{f(n)}{m(\tau_n)^{\lfloor c \log n \rfloor + 1}}
\leq \tau_n \frac{n}{m(\tau_n)^{c \log n}}
= \tau_n \exp \left( (\log n) (1 - c \log m(\tau_n)) \right) \text{.}
\end{align*}
By (\ref{eq tau_n to tau}), for $c > 1/\log m(\tau)$ we thus obtain
\begin{align*}
\lim_{n \rightarrow \infty} \tau_n \frac{f(n)}{m(\tau_n)^{\lfloor c \log n \rfloor + 1}} = 0 \text{,}
\end{align*}
which proves (\ref{lem depth lightning eq essence}).
\end{proof}

\begin{lem} \label{lem depth destruction}
Let
\begin{align*}
J^n := \max \left\{ j \in \{0, \ldots, n\}: \left( \exists z \in S^n_{\tau_n, x} \cap T_{n - j} \, \exists t \in (0, \tau_n]: \eta^n_{t^-, z} > \eta^n_{t, z} \right) \right\} \vee (-1)
\end{align*}
be the ``depth'' of destruction in $S^n_{\tau_n, x}$ up to time $\tau_n$. (On the event $\{ \forall z \in S^n_{\tau_n, x} \, \forall t \in (0, \tau_n]: \eta^n_{t^-, z} \leq \eta^n_{t, z} \}$ we have $J^n = -1$ by definition.) Then we have
\begin{align*}
J^n \stackrel{\mf{P}}{=} \mc{O} \left( f(n) \log n \right) \qquad \text{for } n \rightarrow \infty \text{.}
\end{align*}
\end{lem}

\begin{proof}
Let $c, \tilde c > 0$, $n \in \mb{N}$ and let
\begin{align*}
\mb{F}_{n, \tilde c} := \left\{ \iota^n \geq \frac{1}{\tilde c f(n)}, N^n \leq \tilde c f(n), K^n \leq \tilde c \log n \right\} \text{.}
\end{align*}
By Lemmas \ref{lem distance lightning}, \ref{lem number lightning} and \ref{lem depth lightning}, it suffices to show
\begin{align} \label{lem depth destruction eq essence}
\forall \tilde c > 0: \lim_{c \rightarrow \infty} \limsup_{n \rightarrow \infty} \mf{Q}_n \left[ J^n > c f(n) \log n, \mb{F}_{n, \tilde c} \right] = 0 \text{.}
\end{align}

Let $(Z_i^n)_{i = 1, \ldots, N^n}$ be an enumeration of the sites in $S^n_{\tau_n, x}$ which are hit by ignition up to time $\tau_n$, where we count these sites with multiplicity, i.e.\ for each $z \in S^n_{\tau_n, x}$ the relation $\left| \left\{ i \in \{1, \ldots, N^n\}: Z_i^n = z \right\} \right| = I^n_{\tau_n, z}$ holds. (On the event $\{N^n = 0\}$ the sequence $(Z_i^n)_{i = 1, \ldots, N^n}$ is empty.) For $t \geq 0$ and $z \in T$ let
\begin{align*}
A_{t, z} := |z| - \min \left\{ |w|: w \in S_{t, z} \right\}
\end{align*}
be the difference between the generation of $z$ and the lowest generation which is contained in the cluster of $z$ in the pure growth process at time $t$. (On the event $\{S_{t, z} = \emptyset\}$ we have $A_{t, z} = -\infty$.) Now suppose that $0 \leq t_1 \leq t_2$, $z \in T$ and $k \in \mb{N}$ are given: Using the inclusion $\{A_{t_1, z} \geq k\} \subset \{A_{t_2, z} \geq k\}$ and the fact that the growth processes at different sites are independent, one can show that
\begin{align*}
\mf{P} \left[ A_{t_1, z} \geq k \left| (\sigma_{t_2, w})_{w \in T} \right. \right]
= \left( \frac{1 - e^{-t_1}}{1 - e^{-t_2}} \right)^{k + 1} 1_{\{A_{t_2, z} \geq k\}} \text{.}
\end{align*}
Additionally, let $\mc{S}_{t_2} \mc{I}^n := \sigma \left( (\sigma_{t_2, w})_{w \in T}, (I^n_{t, w})_{t \geq 0, w \in B_n} \right)$ denote the $\sigma$-field generated by the configuration of the pure growth process at time $t_2$ and all ignition processes. Since the growth processes and the ignition processes are independent and since $A_{t_1, z}$ only depends on the growth processes, it follows from the previous equation that
\begin{align} \label{lem depth destruction eq conditional prob}
\mf{P} \left[ A_{t_1, z} \geq k \left| \mc{S}_{t_2} \mc{I}^n \right. \right]
= \left( \frac{1 - e^{-t_1}}{1 - e^{-t_2}} \right)^{k + 1} 1_{\{A_{t_2, z} \geq k\}}
\leq \left( \frac{1 - e^{-t_1}}{1 - e^{-t_2}} \right)^{k + 1} \text{.}
\end{align}

We now relate these preliminaries with the proof of (\ref{lem depth destruction eq essence}): Assume that $n$ is large enough so that $c f(n) \log n \geq \tilde c \log n$ and $\tau_n \geq 1/(\tilde c f(n))$ hold. Then
\begin{align}
\left\{ J^n > c f(n) \log n, \mb{F}_{n, \tilde c} \right\}
&\subset \left\{ \exists i \in \{1, \ldots, N^n\}: (n - |Z^n_i|) +  A_{\tau_n - 1/(\tilde c f(n)), Z^n_i} > c f(n) \log n, \mb{F}_{n, \tilde c} \right\} \nonumber \\
&\subset \left\{ \exists i \in \{1, \ldots, N^n\}: A_{\tau_n - 1/(\tilde c f(n)), Z^n_i} > c f(n) \log n - \tilde c \log n, \mb{F}_{n, \tilde c} \right\} \label{lem depth destruction eq relation A and Z}
\end{align}
holds, where the first inclusion uses $\mb{F}_{n, \tilde c} \subset \left\{ \iota^n \geq 1/(\tilde cf(n)) \right\}$ and the second inclusion is due to the fact that $\mb{F}_{n, \tilde c} \subset \left\{ \forall i \in \{1, \ldots, N^n\}: (n - |Z^n_i|) \leq \tilde c \log n \right\}$. Furthermore, we deduce from (\ref{lem depth destruction eq relation A and Z}) and (\ref{lem depth destruction eq conditional prob}) that
\begin{align*}
&\mf{P} \left[ J^n > c f(n) \log n, \mb{F}_{n, \tilde c} \right] \\
&\hspace{1cm}\leq \mf{E}_{\mf{P}} \left[ \sum_{i = 1}^{N^n} 1_{\{ A_{\tau_n - 1/(\tilde c f(n)), Z^n_i} > c f(n) \log n - \tilde c \log n \}} 1_{\mb{F}_{n, \tilde c}} \right] \\
&\hspace{1cm}= \mf{E}_{\mf{P}} \left[ \sum_{i = 1}^{N^n} \sum_{z \in S^n_{\tau_n, x}} \mf{P} \left[ \left. A_{\tau_n - 1/(\tilde c f(n)), z} \geq \lfloor c f(n) \log n - \tilde c \log n \rfloor + 1 \right| \mc{S}_{\tau_n} \mc{I}^n \right] 1_{\{Z^n_i = z\}} 1_{\mb{F}_{n, \tilde c}} \right] \\
&\hspace{1cm}\leq \tilde c f(n) \left( \frac{1 - e^{-\tau_n + 1/(\tilde c f(n))}}{1 - e^{-\tau_n}} \right)^{\lfloor c f(n) \log n - \tilde c \log n \rfloor + 2}
\end{align*}
holds. Now let $n$ be large enough so that $f(n) \leq n$ holds (which is possible by (\ref{eq condition on f})). Then
\begin{align*}
\tilde c f(n) \left( \frac{1 - e^{-\tau_n + 1/(\tilde c f(n))}}{1 - e^{-\tau_n}} \right)^{\lfloor c f(n) \log n - \tilde c \log n \rfloor + 2}
\leq \tilde cn \left( \frac{1 - e^{-\tau_n + 1/(\tilde c f(n))}}{1 - e^{-\tau_n}} \right)^{c f(n) \log n - \tilde c \log n} \text{.}
\end{align*}
In order to determine the behaviour of the last term for $n \rightarrow \infty$, we rewrite it as
\begin{align}
&n \left( \frac{1 - e^{-\tau_n + 1/(\tilde c f(n))}}{1 - e^{-\tau_n}} \right)^{c f(n) \log n - \tilde c \log n} \nonumber \\
&\hspace{1cm}= \exp \left( \log n + \left( c f(n) \log n - \tilde c \log n \right) \log \left( 1 - \frac{e^{1/(\tilde c f(n))} - 1}{e^{\tau_n} - 1} \right) \right) \nonumber \\
&\hspace{1cm}= \exp \left( (\log n) \left( 1 + \left( c - \frac{\tilde c}{f(n)} \right) f(n) \log \left( 1 - \frac{e^{1/(\tilde c f(n))} - 1}{e^{\tau_n} - 1} \right) \right) \right) \text{.} \label{lem depth destruction eq rewritten expression}
\end{align}
Since $f(n) \rightarrow \infty$ and $\tau_n \rightarrow \tau$ for $n \rightarrow \infty$ (see (\ref{eq condition on f}) and (\ref{eq tau_n to tau})), we calculate
\begin{align*}
\lim_{n \rightarrow \infty} \frac{e^{\tau_n} - 1}{e^{1/(\tilde c f(n))} - 1} \log \left( 1 - \frac{e^{1/(\tilde c f(n))} - 1}{e^{\tau_n} - 1} \right)
&= \lim_{y \downarrow 0} \frac{\log \left( 1 - y \right)}{y} = -1 \text{,} \\
\lim_{n \rightarrow \infty} \tilde cf(n) \left( e^{1/(\tilde c f(n))} - 1 \right)
&= \lim_{y \downarrow 0} \frac{e^y - 1}{y} = 1 \text{,} \\
\lim_{n \rightarrow \infty} \frac{1}{\tilde c \left( e^{\tau_n} - 1 \right)}
&= \frac{1}{\tilde c \left( e^{\tau} - 1 \right)} \text{.}
\end{align*}
Multiplying these equations yields
\begin{align} \label{lem depth destruction eq crucial limit}
\lim_{n \rightarrow \infty} f(n) \log \left( 1 - \frac{e^{1/(\tilde c f(n))} - 1}{e^{\tau_n} - 1} \right)
= \frac{-1}{\tilde c \left( e^{\tau} - 1 \right)} \text{.}
\end{align}
From (\ref{lem depth destruction eq rewritten expression}) and (\ref{lem depth destruction eq crucial limit}) we conclude that for $c > \tilde c \left( e^{\tau} - 1 \right)$ we have
\begin{align*}
\lim_{n \rightarrow \infty} \tilde cn \left( \frac{1 - e^{-\tau_n + 1/(\tilde c f(n))}}{1 - e^{-\tau_n}} \right)^{c f(n) \log n - \tilde c \log n} = 0 \text{,}
\end{align*}
which proves (\ref{lem depth destruction eq essence}).
\end{proof}

\begin{proof}[Proof of (\ref{eq thm general convergence part 1}) and (\ref{eq thm general convergence part 2}) (and hence of Theorem \ref{thm general convergence})]
Equation (\ref{eq thm general convergence part 1}) is an immediate consequence of Lemma \ref{lem depth destruction} and the fact that $f(n) \log n \ll n$ for $n \rightarrow \infty$ by (\ref{eq condition on f}).

Proof of (\ref{eq thm general convergence part 2}): Let $c > 0$, $n \in \mb{N}$ and let
\begin{align*}
\mb{G}_{n, c} := \left\{ \frac{|S^n_{\tau_n, x}|}{m(\tau_n)^n} \geq \frac{1}{c}, N^n \leq cf(n), J^n \leq c f(n) \log n \right\} \text{.}
\end{align*}
By Proposition \ref{prop replacement Kesten-Stigum}, equation (\ref{prop replacement Kesten-Stigum eq lower bound}), Lemma \ref{lem number lightning} und Lemma \ref{lem depth destruction}, it suffices to show
\begin{align} \label{eq thm general convergence part 2 essence}
\forall c > 0: \lim_{n \rightarrow \infty} \mf{Q}_n \left[ \left. \exists t \in (\tau_n, \tau_n + \delta): \eta^n_{t^-, x} > \eta^n_{t, x} \right| \mb{G}_{n, c} \right] = 1 \text{.}
\end{align}

We first observe that
\begin{align} \label{eq difference Sn and Cn}
\left| S^n_{\tau_n, x} \setminus C^n_{\tau_n, x} \right|
\leq N^n \frac{r^{J^n + 1} - 1}{r - 1}
\end{align}
holds: In the case where $J^n = -1$ we have $C^n_{\tau_n, x} = S^n_{\tau_n, x}$ so that (\ref{eq difference Sn and Cn}) holds indeed. In the case where $J^n \geq 0$ the cluster $C^n_{\tau_n, x}$ can only differ from $S^n_{\tau_n, x}$ in the maximal subtrees of $S^n_{\tau_n, x}$ whose roots are in $T_{n - J^n}$. Each of these maximal subtrees can have at most $\sum_{j = 0}^{J^n} r^j = \frac{r^{J^n + 1} - 1}{r - 1}$ vertices. Moreover, since these subtrees are disconnected, at most $N^n$ of them can have been affected by destruction up to time $\tau_n$. This proves (\ref{eq difference Sn and Cn}) in the second case. On the event $\mb{G}_{n, c}$, we hence have
\begin{align*}
\left| S^n_{\tau_n, x} \setminus C^n_{\tau_n, x} \right|
&\leq cf(n) \frac{r^{c f(n) \log n + 1} - 1}{r - 1} \\
&\leq cf(n) r^{c f(n) \log n + 1}
\end{align*}
and
\begin{align}
\left| C^n_{\tau_n, x} \right|
&= \left| S^n_{\tau_n, x} \right| - \left| S^n_{\tau_n, x} \setminus C^n_{\tau_n, x} \right| \nonumber \\
&\geq \frac{1}{c} m(\tau_n)^n - cf(n) r^{c f(n) \log n + 1} \text{.} \label{eq estimate Cn}
\end{align}
For $t \geq 0$, let $\mc{F}^n_t := \sigma((G_{s, w})_{0 \leq s \leq t, w \in T}, (I^n_{s, w})_{0 \leq s \leq t, w \in B_n})$ denote the $\sigma$-field generated by the growth and ignition processes up to time $t$. We then deduce
\begin{align}
&\mf{Q}_n \left[ \exists t \in (\tau_n, \tau_n + \delta): \eta^n_{t^-, x} > \eta^n_{t, x}, \mb{G}_{n, c} \right] \nonumber \\
&\hspace{1cm} \geq \mf{Q}_n \left[ \exists z \in C^n_{\tau_n, x}: I^n_{\tau_n + \delta, z} > I^n_{\tau_n, z}, \mb{G}_{n, c} \right] \nonumber \\
&\hspace{1cm} = \mf{E}_{\mf{Q}_n} \left[ \mf{Q}_n \left[ \left. \exists z \in C^n_{\tau_n, x}: I^n_{\tau_n + \delta, z} > I^n_{\tau_n, z} \right| \mc{F}^n_{\tau_n} \right] 1_{\mb{G}_{n, c}} \right] \nonumber \\
&\hspace{1cm} = \mf{E}_{\mf{Q}_n} \left[ \left( 1 - \exp \left( -\delta \lambda(n) \left| C^n_{\tau_n, x} \right| \right) \right) 1_{\mb{G}_{n, c}} \right] \nonumber \\
&\hspace{1cm} \geq \left( 1 - \exp \left( -\delta \lambda(n) \left( \frac{1}{c} m(\tau_n)^n - cf(n) r^{c f(n) \log n + 1} \right) \right) \right) \mf{Q}_n \left[ \mb{G}_{n, c} \right] \text{,} \label{eq estimate Gnc}
\end{align}
where the last inequality follows from (\ref{eq estimate Cn}). In order to determine the behaviour of the exponential argument for $n \rightarrow \infty$, we consider the two summands separately: For the first summand we clearly have
\begin{align} \label{eq limit first summand}
\lim_{n \rightarrow \infty} \lambda(n) m(\tau_n)^n
= \lim_{n \rightarrow \infty} f(n)
= \infty
\end{align}
(see (\ref{eq condition on f})). The second summand can be rewritten as
\begin{align*}
\lambda(n) f(n) r^{c f(n) \log n}
&= f(n)^2 \frac{r^{c f(n) \log n}}{m(\tau_n)^n} \\
&= \exp \left( 2 \log f(n) + (c \log r) f(n) \log n - (\log m(\tau_n)) n \right) \\
&= \exp \left( n \left( 2 \frac{\log f(n)}{n} + (c \log r) \frac{f(n) \log n}{n} - \log m(\tau_n) \right) \right) \text{.}
\end{align*}
By (\ref{eq condition on f}), the function $f$ satisfies $f(n) \ll n/\log n$ for $n \rightarrow \infty$, and this also implies $\log f(n) \ll n$ for $n \rightarrow \infty$. Using these asymptotics and recalling (\ref{eq tau_n to tau}), we thus conclude
\begin{align*}
\lim_{n \rightarrow \infty} \left( 2 \frac{\log f(n)}{n} + (c \log r) \frac{f(n) \log n}{n} - \log m(\tau_n) \right) = -\log m(\tau)
\end{align*}
and
\begin{align} \label{eq limit second summand}
\lim_{n \rightarrow \infty} \lambda(n) f(n) r^{c f(n) \log n} = 0 \text{.}
\end{align}
Putting (\ref{eq estimate Gnc}), (\ref{eq limit first summand}) and (\ref{eq limit second summand}) together yields the proof of (\ref{eq thm general convergence part 2 essence}).
\end{proof}

\section{Proof of Theorem~\ref{thm refined convergence}} \label{sec proof refined convergence}

Consider the same setup as in Section \ref{subsec proof general convergence}; additionally, let $g: \mb{N} \rightarrow (0, \infty)$ be defined as in (\ref{equation g}). By assumption, $g$ satisfies $\sqrt[n]{g(n)} \rightarrow 1$ for $n \rightarrow \infty$.

Part (i): Suppose that $g$ also satisfies $g(n) \ll n/\log n$ for $n \rightarrow \infty$. Then the function $f$ of Section~\ref{subsec proof general convergence} can be chosen in such a way that for $n$ large enough $f(n) \geq g(n)$ holds. Since $m^{-1}$ is monotone increasing, we conclude that for $n$ large enough
\begin{align*}
\tau_n
= m^{-1} \left( m(\tau) \sqrt[n]{\frac{f(n)}{g(n)}} \right)
\geq \tau
\end{align*}
holds. By (\ref{eq tau_n to tau}) we also have $\tau_n \rightarrow \tau$ for $n \rightarrow \infty$. Theorem~\ref{thm refined convergence} (i) therefore follows from (\ref{eq thm general convergence with tau_n}).

Part (ii): Suppose that there exists $\alpha \in (0, 1)$ such that $g$ satisfies $g(n) \gg \exp(n^{\alpha})$ for $n \rightarrow \infty$. Choose $\beta, \gamma \in (0, 1)$ such that $0 < \beta < \alpha$ and $0 < 1 - \beta < \gamma$ hold. Take the function $f$ of Section~\ref{subsec proof general convergence} to be $f(n) := n^{\gamma}$, $n \in \mb{N}$. Clearly, (\ref{eq condition on f}) is satisfied for this choice of $f$, and for $n$ large enough we have $f(n) \leq g(n)$. Hence, similar arguments as above show that for $n$ large enough $\tau_n \leq \tau$ holds. Again, by (\ref{eq tau_n to tau}) we also have $\tau_n \rightarrow \tau$ for $n \rightarrow \infty$. Using these facts and arguing analogously to Section~\ref{subsec proof general convergence}, we conclude that for Theorem~\ref{thm refined convergence} (ii) it suffices to prove
\begin{align} \label{eq thm refined convergence part 2}
\lim_{n \rightarrow \infty} \mf{Q}_n \left[ \exists t \in (\tau_n, \tau): \eta^n_{t^-, x} > \eta^n_{t, x} = 0 \right] = 1
\end{align}
for $x \in T$, where $\mf{Q}_n$ is defined as in Section~\ref{subsec proof general convergence}. In Section~\ref{subsec proof general convergence}, we deduced that equation (\ref{eq thm general convergence part 2}) follows from (\ref{eq limit first summand}) and (\ref{eq limit second summand}); in exactly the same way it can be shown that equation (\ref{eq thm refined convergence part 2}) follows from
\begin{align} \label{eq limit first summand refined}
\lim_{n \rightarrow \infty} \left( \tau - \tau_n \right) f(n)
= \infty
\end{align}
and
\begin{align} \label{eq limit second summand refined}
\lim_{n \rightarrow \infty} \left( \tau - \tau_n \right) \lambda(n) f(n) r^{c f(n) \log n} = 0 \text{.}
\end{align}
Now (\ref{eq limit second summand refined}) is an immediate consequence of (\ref{eq limit second summand}). It thus remains to prove (\ref{eq limit first summand refined}).

To this end we first rewrite $\tau - \tau_n$ as
\begin{align*}
\tau - \tau_n
= \log \left( \frac{r - m(\tau) \sqrt[n]{f(n)/g(n)}}{r - m(\tau)} \right)
= \log \left( 1 + \frac{m(\tau)}{r - m(\tau)} \left( 1 - \sqrt[n]{\frac{f(n)}{g(n)}} \right) \right) \text{.}
\end{align*}
Since $\sqrt[n]{f(n)} \rightarrow 1$ and $\sqrt[n]{g(n)} \rightarrow 1$ for $n \rightarrow \infty$ (the first limit follows from (\ref{eq condition on f}), the second limit holds by assumption), we conclude that
\begin{align*}
\lim_{n \rightarrow \infty} \left( 1 - \sqrt[n]{\frac{f(n)}{g(n)}} \right)^{-1} \left( \tau - \tau_n \right)
= \lim_{y \downarrow 0} y^{-1} \log \left( 1 + \frac{m(\tau)}{r - m(\tau)} y \right)
= \frac{m(\tau)}{r - m(\tau)} > 0
\end{align*}
holds. For $n$ large enough, we also have $\frac{f(n)}{g(n)} \leq \frac{n^{\gamma}}{\exp(n^{\alpha})} \leq \frac{1}{\exp(n^{\beta})}$ and hence
\begin{align*}
f(n) \left( 1 - \sqrt[n]{\frac{f(n)}{g(n)}} \right)
\geq n^{\gamma} \left( 1 - \exp \left( -n^{\beta - 1} \right) \right) \text{.}
\end{align*}
Moreover, the limit $n \rightarrow \infty$ of the last term is given by
\begin{align*}
\lim_{n \rightarrow \infty} n^{\gamma} \left( 1 - \exp \left( -n^{\beta - 1} \right) \right)
= \lim_{n \rightarrow \infty} n^{\gamma + \beta - 1} \cdot \lim_{n \rightarrow \infty} n^{-(\beta - 1)} \left( 1 - \exp \left( -n^{\beta - 1} \right) \right)
= \infty \text{.}
\end{align*}
This yields the proof of (\ref{eq limit first summand refined}).

\section{Proof of Proposition \ref{prop critical epsilon}} \label{sec proof critical epsilon}

Let $\tau \in (t_c, \infty)$, let $\epsilon > 0$ and let $(\rho_{t, z}, G_{t, z})_{0 \leq t \leq \tau + \epsilon, z \in T}$ be a self-destructive percolation process on $T$ with parameters $\tau$ and $\epsilon$ under some probability measure $\mf{P}$. So far we have parametrized self-destructive percolation in terms of the length of the time intervals $[0, \tau)$ and $[\tau, \tau + \epsilon]$. For the proof of Proposition \ref{prop critical epsilon}, however, it will be more convenient to parametrize the final configuration $\rho_{\tau + \epsilon} := (\rho_{\tau + \epsilon, z})_{z \in T}$ in terms of the Bernoulli probabilities $p := 1 - e^{-\tau}$ and $\delta := 1 - e^{-\epsilon}$ for growth at a fixed vertex in the time intervals $[0, \tau)$ and $[\tau, \tau + \epsilon]$, respectively. We therefore use the following alternative notation (which follows along the lines of \cite{berg2004destructive}):

Let $X_v$, $v \in T$, and $Y_v$, $v \in T$, be independent $\{0, 1\}$-valued random variables under some probability measure $\mf{P}_{p, \delta}$ such that
\begin{align*}
\mf{P}_{p, \delta} \left[ X_v = 1 \right] &= p \text{,} & \mf{P}_{p, \delta} \left[ X_v = 0 \right] &= 1 - p \text{,} \\
\mf{P}_{p, \delta} \left[ Y_v = 1 \right] &= \delta \text{,} & \mf{P}_{p, \delta} \left[ Y_v = 0 \right] &= 1 - \delta
\end{align*}
for all $v \in T$. Let $X := (X_v)_{v \in T}$, $Y := (Y_v)_{v \in T}$ and define $X^* = (X^*_v)_{v \in T}$, $Z = (Z_v)_{v \in T}$ by
\begin{align*}
X_v^* :=
\begin{cases}
1 & \text{if $X_v = 1$ and the cluster of $v$ in $X$ is finite,} \\
0 & \text{otherwise,}
\end{cases}
\end{align*}
and
\begin{align*}
Z_v := X_v^* \vee Y_v \text{.}
\end{align*}
Then the distribution of the final configuration $\rho_{\tau + \epsilon}$ under $\mf{P}$ and the distribution of $Z$ under $\mf{P}_{p, \delta}$ are clearly identical.

Let
\begin{align*}
\theta(p) := \mf{P}_{p, \delta} \left[ \text{the cluster of $\emptyset$ in $X$ is infinite} \right]
\end{align*}
be the probability that the root $\emptyset$ is in an infinite cluster after the first step of self-destructive percolation (i.e.\ in independent site percolation on $T$ with parameter $p$), and let
\begin{align*}
\theta(p, \delta) := \mf{P}_{p, \delta} \left[ \text{the cluster of $\emptyset$ in $Z$ is infinite} \right]
\end{align*}
be the probability that the root $\emptyset$ is in an infinite cluster in the final configuration of self-destructive percolation. Using the fact that the final configuration $Z$ is positively associated (\cite{berg2004destructive}, Sections 2.2 and 2.3), it is easy to see that the equivalence
\begin{align*}
\mf{P}_{p, \delta} \left[ \text{$Z$ contains an infinite cluster} \right] = 0
\Leftrightarrow \theta(p, \delta) = 0
\end{align*}
holds. For the proof of Proposition \ref{prop critical epsilon} it therefore suffices to prove the following proposition, where $p_c := \frac{1}{r} = 1 - e^{-t_c}$ denotes the critical probability of independent site percolation on $T$:

\begin{prop} \label{prop critical delta}
For all $p \in (p_c, 1)$ there exists $\delta \in (0, 1)$ such that $\theta(p, \delta) = 0$.
\end{prop}

Proposition \ref{prop critical delta} is a generalization of a result by J.\ van den Berg and R.\ Brouwer (\cite{berg2004destructive}, Theorem~5.1), who proved the following statement for the case where $T$ is the binary tree (i.e.\ $r = 2$): If $p \in (p_c, 1)$ and $\delta > 0$ satisfies
\begin{align} \label{eq delta binary tree}
p \left( 1 - \delta \right) \geq p_c \text{,}
\end{align}
then $\theta(p, \delta) = 0$. Our proof of Proposition \ref{prop critical delta} for general $r$ is based on the same principal ideas as \cite{berg2004destructive} but eventually takes a different route due to the occurrence of higher order terms for $r \geq 3$. Although these terms turn out to be asymptotically negligible, they are the reason why for $r \geq 3$ we do not obtain an explicit condition on $\delta$ like (\ref{eq delta binary tree}).

We first prove a weaker version of Proposition \ref{prop critical delta}:

\begin{lem} \label{lem critical delta}
For all $p \in (p_c, 1)$ we have $\lim_{\delta \downarrow 0} \theta(p, \delta) = 0$.
\end{lem}

\begin{proof}[Proof of Lemma \ref{lem critical delta}]
Let $p \in (p_c, 1)$ and $\delta \in (0, 1)$. By distinguishing whether or not the root $\emptyset$ is in an infinite cluster after the first step of self-destructive percolation we obtain the inequality
\begin{align*}
\theta(p, \delta)
&\leq \mf{P}_{p, \delta} \left[ \text{the cluster of $\emptyset$ in $X$ is infinite}, Y_{\emptyset} = 1 \right] \\
&\hspace{4mm} + \mf{P}_{p, \delta} \left[ \text{the cluster of $\emptyset$ in $X$ is finite}, \text{the cluster of $\emptyset$ in $X \vee Y$ is infinite} \right] \\
&= \theta(p) \delta + \left( \theta(p + (1 - p)\delta) - \theta(p) \right) \text{.}
\end{align*}
Since $\theta(\,\cdot\,)$ is continuous, the last expression tends to zero for $\delta \downarrow 0$, which proves the lemma.
\end{proof}

\begin{proof}[Proof of Proposition \ref{prop critical delta}]
Suppose that Proposition \ref{prop critical delta} is not true. Then there exists $p_0 \in (p_c, 1)$ such that for all $\delta \in (0, 1)$ we have $\theta(p_0, \delta) > 0$. In fact, even the stronger statement
\begin{align} \label{prop critical delta eq p_0}
\forall p \in (p_c, p_0] \, \forall \delta \in (0, 1): \theta(p, \delta) > 0
\end{align}
is true. This is due to the fact that if $p_1, p_2 \in (p_c, 1)$ and $\delta_1, \delta_2 \in (0, 1)$ satisfy $p _1 \geq p_2$ and $p_1 + (1 - p_1) \delta_1 = p_2 + (1 - p_2) \delta_2$, then $\theta(p_1, \delta_1) \leq \theta(p_2, \delta_2)$ holds (see \cite{berg2004destructive}, Lemma~2.3). We will show that (\ref{prop critical delta eq p_0}) leads to a contradiction.

\begin{figure}
\centering
\psset{unit=0.4cm}
\begin{pspicture*}(-15.2,-1.5)(15.2,17.2)
\def\lw{0.06}
\def\lwh{0.03}
\def\ds{1.5}
\psline[linewidth=\lw,linestyle=solid,linecolor=black](0,0)(-9,6)
\psline[linewidth=\lw,linestyle=solid,linecolor=black](0,0)(0,6)
\psline[linewidth=\lw,linestyle=solid,linecolor=black](0,0)(9,6)
\multido{\ixb=-9+9,\ixt=-12+9}{3}{
\psline[linewidth=\lw,linestyle=solid,linecolor=black](\ixb,6)(\ixt,12)
}
\multido{\ix=-9+9}{3}{
\psline[linewidth=\lw,linestyle=solid,linecolor=black](\ix,6)(\ix,12)
}
\multido{\ixb=-9+9,\ixt=-6+9}{3}{
\psline[linewidth=\lw,linestyle=solid,linecolor=black](\ixb,6)(\ixt,12)
}
\multido{\ixb=-12+3,\rxt=-12.5+3.0}{9}{
\psline[linewidth=\lw,linestyle=solid,linecolor=black](\ixb,12)(\rxt,15)
}
\multido{\ix=-12+3}{9}{
\psline[linewidth=\lw,linestyle=solid,linecolor=black](\ix,12)(\ix,15)
}
\multido{\ixb=-12+3,\rxt=-11.5+3.0}{9}{
\psline[linewidth=\lw,linestyle=solid,linecolor=black](\ixb,12)(\rxt,15)
}
\psdots[dotscale=\ds,dotstyle=*,fillstyle=solid,linecolor=black](0,0)
\multido{\ix=-9+9}{3}{
\psdots[dotscale=\ds,dotstyle=*,fillstyle=solid,linecolor=black](\ix,6)
}
\multido{\ix=-12+3}{9}{
\psdots[dotscale=\ds,dotstyle=*,fillstyle=solid,linecolor=black](\ix,12)
}
\psbezier[linewidth=\lwh,linestyle=dashed,linecolor=black](-4,17)(-2.8,0)(2.8,0)(4,17)
\psbezier[linewidth=\lwh,linestyle=dashed,linecolor=black](-13,17)(-11.8,0)(-6.2,0)(-5,17)
\psbezier[linewidth=\lwh,linestyle=dashed,linecolor=black](13,17)(11.8,0)(6.2,0)(5,17)
\psbezier[linewidth=\lwh,linestyle=dashed,linecolor=black](-15,17)(-12,-7.5)(12,-7.5)(15,17)
\multido{\rx=-12.1+3.0}{9}{
\rput[lt](\rx,17){$\vdots$}
}
\rput[lt](0.6,0.1){$\emptyset$}
\rput[lt](-8.4,6.4){$1$}
\rput[lt](0.6,6.4){$2$}
\rput[lt](9.6,6.4){$3$}
\rput[lb](13.8,16){$T$}
\rput[lb](-8,16){$T^{(1)}$}
\rput[lb](1,16){$T^{(2)}$}
\rput[lb](10,16){$T^{(3)}$}
\end{pspicture*}
\caption{Illustration of the case $r = 3$ - the $3$-regular rooted tree $T$ with the root $\emptyset$, its children $1, 2, 3$ and the $3$-regular rooted subtrees $T^{(1)}, T^{(2)}, T^{(3)}$}
\label{fig tree and subtrees}
\end{figure}
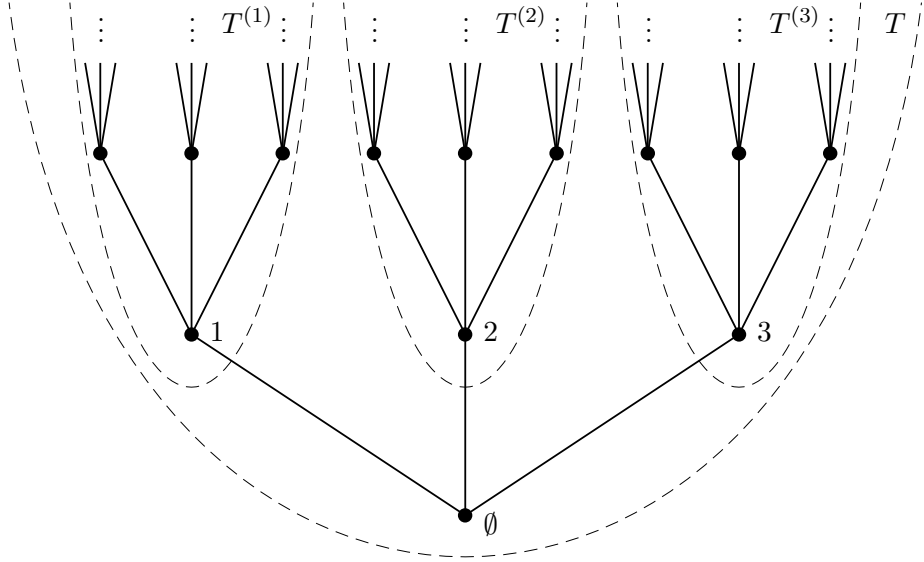

Let $p \in (p_c, p_0]$, $\delta \in (0, 1)$ and define the probability measure $\mf{P}_{p, \delta}$ and the random configurations $X$, $Y$, $X^*$, $Z$ as above (at the beginning of Section \ref{sec proof critical epsilon}). We will derive an inequality for $\theta(p, \delta)$ by exploiting the recursive structure of the tree $T$.
So let us denote the $r$ children of the root $\emptyset$ by $1, \ldots, r$. For $i = 1, \ldots, r$, let $T^{(i)}$ be the $r$-regular rooted subtree of $T$ which has $i$ as its root (see Figure \ref{fig tree and subtrees} for an illustration of the case $r = 3$). As before, we will use the term $T^{(i)}$ both for the graph and its vertex set. Let $X^{(i)} := (X_v)_{v \in T^{(i)}}$ and $Y^{(i)} := (Y_v)_{v \in T^{(i)}}$ be the configurations we obtain when we restrict $X$ and $Y$ to the subtree $T^{(i)}$. Moreover, let $X^{*(i)} = (X^{*(i)}_v)_{v \in T^{(i)}}$ and $Z^{(i)} = (Z^{(i)}_v)_{v \in T^{(i)}}$ be the corresponding configurations for self-destructive percolation on $T^{(i)}$, i.e.
\begin{align*}
X_v^{*(i)} :=
\begin{cases}
1 & \text{if $X_v = 1$ and the cluster of $v$ in $X^{(i)}$ is finite,} \\
0 & \text{otherwise,}
\end{cases}
\end{align*}
and
\begin{align*}
Z_v^{(i)} := X_v^{*(i)} \vee Y_v \text{.}
\end{align*}
Then the quadruples of configurations $(X^{(i)}, Y^{(i)}, X^{*(i)}, Z^{(i)})$, $i = 1, \ldots, r$, are independent and have the same distribution as $(X, Y, X^{*}, Z)$.

Now consider the events
\begin{align*}
A := \left\{ X_{\emptyset} \vee Y_{\emptyset} = 1, \exists i \in \{1, \ldots, r\}: \text{the cluster of $i$ in $Z^{(i)}$ is infinite} \right\}
\end{align*}
and
\begin{align*}
B := \Bigl\{ X_{\emptyset} = 1, Y_{\emptyset} = 0, \exists i, j \in \{1, \ldots, r\}: i \not= j, \text{the cluster of $i$ in $Z^{(i)}$ is infinite}, \\
\text{the cluster of $j$ in $X^{(j)}$ is infinite} \Bigr\} \text{.}
\end{align*}
Since these events satisfy the inclusions
\begin{align*}
\left\{ \text{the cluster of $\emptyset$ in $Z$ is infinite} \right\} &\subset A \text{,} \\
\left\{ \text{the cluster of $\emptyset$ in $Z$ is finite} \right\} &\supset B \text{,} \\
B &\subset A \text{,}
\end{align*}
we have
\begin{align} \label{prop critical delta eq inequality A B}
\theta(p, \delta) \leq \mf{P}_{p, \delta} \left[ A \right] - \mf{P}_{p, \delta} \left[ B \right] \text{.}
\end{align}
From the definition of $A$ we readily deduce
\begin{align} \label{prop critical delta eq equality A}
\begin{alignedat}{2}
\mf{P}_{p, \delta} \left[ A \right]
&= \makebox[0pt][l]{$\displaystyle  \left( p + (1 - p)\delta \right) \left( 1 - \left( 1 - \theta(p, \delta) \right)^r \right)$} \\
&= \left( p + (1 - p)\delta \right) \cdot r \theta(p, \delta) \,+\, && \mc{O}(\theta(p, \delta)^2) \\
& && \text{for $\delta \downarrow 0$, uniformly for $p \in (p_c, p_0]$.}
\end{alignedat}
\end{align}
In order to calculate $\mf{P}_{p, \delta} \left[ B \right]$, we define 
\begin{align*}
D_i := \left\{ \text{the cluster of $i$ in $Z^{(i)}$ is infinite}, \exists j \in \{1, \ldots, r\} \setminus \{i\}: \text{the cluster of $j$ in $X^{(j)}$ is infinite} \right\}
\end{align*}
for $i = 1, \ldots, r$ and rewrite $B$ as
\begin{align} \label{prop critical delta eq B Di}
B = \left\{ X_{\emptyset} = 1, Y_{\emptyset} = 0 \right\} \cap \bigcup_{i = 1}^r D_i \text{.}
\end{align}
For $i \in \{1, \ldots, r\}$ the definition of $D_i$ implies
\begin{align*}
\mf{P}_{p, \delta} \left[ D_i\right]
= \theta(p, \delta) \left( 1 - \left( 1 - \theta(p) \right)^{r - 1} \right) \text{,}
\end{align*}
and for $k \in \{2, \ldots, r\}$ and $1 \leq i_1 < \ldots < i_k \leq r$ we have the upper bound
\begin{align*}
\mf{P}_{p, \delta} \left[ D_{i_1} \cap \ldots \cap D_{i_k} \right]
&\leq \theta(p, \delta)^k \text{.}
\end{align*}
Hence equation (\ref{prop critical delta eq B Di}) yields
\begin{align} \label{prop critical delta eq equality B}
\begin{alignedat}{2}
\mf{P}_{p, \delta} \left[ B \right]
&= \makebox[0pt][l]{$\displaystyle p(1 - \delta) \Biggl( \sum_{i = 1}^r \mf{P}_{p, \delta} \left[ D_i \right] + \sum_{k = 2}^r \sum_{1 \leq i_1 < \ldots < i_k \leq r} (-1)^{k + 1} \, \mf{P}_{p, \delta} \left[ D_{i_1} \cap \ldots \cap D_{i_k} \right] \Biggr)$} \\
&= p(1 - \delta) \cdot r \theta(p, \delta) \left( 1 - \left( 1 - \theta(p) \right)^{r - 1} \right) \,+\, &&\mc{O}(\theta(p, \delta)^2) \\
& && \text{for $\delta \downarrow 0$, uniformly for $p \in (p_c, p_0]$.}
\end{alignedat}
\end{align}
Inserting (\ref{prop critical delta eq equality A}) and (\ref{prop critical delta eq equality B}) into the inequality (\ref{prop critical delta eq inequality A B}) and dividing both sides by $\theta(p, \delta)$ (which is possible because of our assumption (\ref{prop critical delta eq p_0})), we obtain
\begin{align*}
1
\leq \left( p + (1 - p)\delta \right) \cdot r - p(1 - \delta) \cdot r \left( 1 - \left( 1 - \theta(p) \right)^{r - 1} \right) \,+\, &\mc{O}(\theta(p, \delta)) \\
&\text{for $\delta \downarrow 0$, uniformly for $p \in (p_c, p_0]$.}
\end{align*}
Finally, letting $\delta$ tend to zero and using Lemma \ref{lem critical delta}, we get
\begin{align*}
1 \leq pr \left( 1 - \theta(p) \right)^{r - 1} \text{.}
\end{align*}
In the remainder of the proof we show that this inequality leads to a contradiction when $p$ tends to $p_c$. Expanding the right side of the inequality in powers of $\theta(p)$, we obtain
\begin{align} \label{prop critical delta eq thetap 1}
1 \leq pr \left( 1 - (r - 1)\theta(p) \right) + \mc{O}(\theta(p)^2) \qquad \text{for } p \downarrow p_c \text{.}
\end{align}
On the other hand, the recursive structure of the tree $T$ implies
\begin{align*}
\theta(p)
&= p \left( 1 - \left( 1 - \theta(p) \right)^r \right) \\
&= p \left( r\theta(p) - \frac{1}{2}r(r - 1)\theta(p)^2 \right) + \mc{O} \left( \theta(p)^3 \right) \qquad \text{for } p \downarrow p_c \text{.}
\end{align*}
Dividing both sides by $\theta(p)$ (which is positive for $p \in (p_c, p_0]$) gives
\begin{align} \label{prop critical delta eq thetap 2}
1 = pr \left( 1 - \frac{1}{2}(r - 1)\theta(p) \right) + \mc{O} \left( \theta(p)^2 \right) \qquad \text{for } p \downarrow p_c \text{.}
\end{align}
Subtracting (\ref{prop critical delta eq thetap 2}) from (\ref{prop critical delta eq thetap 1}) and dividing by $\theta(p)$ again then leads to the inequality
\begin{align*}
0 \leq - \frac{1}{2}pr(r - 1) + \mc{O} \left( \theta(p) \right) \qquad \text{for } p \downarrow p_c \text{.}
\end{align*}
But since $\theta(p) \rightarrow 0$ for $p \downarrow p_c$, this produces a contradiction.
\end{proof}

\paragraph{Acknowledgement.}
I am grateful to Franz Merkl for helpful discussions and remarks. This work was supported by a scholarship from the Studienstiftung des deutschen Volkes.

\bibliography{forest_fire_tree}
\bibliographystyle{alpha}

\end{document}